\documentclass[11pt, reqno]{amsart}

\usepackage[a4paper,
            top=30mm,bottom=30mm,
            left=35mm,right=35mm]{geometry}

\usepackage{amsmath, amsthm, amsfonts, amssymb, amscd, mathtools, bm} 
\usepackage{xcolor} 
\usepackage[unicode=true]{hyperref} 
\usepackage[shortlabels]{enumitem} 
\usepackage[T1]{fontenc}
\usepackage{empheq}
\usepackage{placeins}

\theoremstyle{plain}
\newtheorem{theorem}{Theorem}[section]
\newtheorem{lemma}[theorem]{Lemma}

\newtheorem{proposition}[theorem]{Proposition}

\theoremstyle{definition}
\newtheorem{definition}[theorem]{Definition}
\newtheorem{example}[theorem]{Example}

\theoremstyle{remark}
\newtheorem{remark}[theorem]{Remark}

\numberwithin{equation}{section}

\newcommand{\eps}{\varepsilon}

\newcommand{\dx}{\mathrm{d} x}

\newcommand{\dz}{\mathrm{d} z}
\newcommand{\dt}{\mathrm{d} t}
\newcommand{\ds}{\mathrm{d} s}
\newcommand{\dm}{\mathrm{d} \mu}
\newcommand{\rd}{\mathrm{d}}
\newcommand{\loc}{\mathrm{loc}}
\newcommand{\X}{\mathrm{X}}
\newcommand{\N}{\mathbb{N}}
\newcommand{\R}{\mathbb{R}}

\newcommand{\F}{\mathrm{F}}

\newcommand{\idot}{\! \cdot \!}

\DeclareMathOperator{\supp}{supp}
\DeclareMathOperator{\Div}{div}

\title[]{\boldmath An $L^1$-theory for $p$-Schr\"{o}dinger equations\\ with confinement in measure}

\author[]{Nuno J. Alves}
\author[]{Jos\'e Miguel Urbano}

\address[N. J. Alves]{
      Applied Mathematics and Computational Sciences (AMCS), Computer, Electrical and Mathematical Sciences and Engineering Division (CEMSE), King Abdullah University of Science and Technology (KAUST), Thuwal, 23955-6900, Kingdom of Saudi Arabia.}
\email{nuno.januarioalves@kaust.edu.sa}

\address[J.M. Urbano]{
      Applied Mathematics and Computational Sciences (AMCS), Computer, Electrical and Mathematical Sciences and Engineering Division (CEMSE), King Abdullah University of Science and Technology (KAUST), Thuwal, 23955-6900, Kingdom of Saudi Arabia and CMUC, Department of Mathematics, University of Coimbra, 3000-143 Coimbra, Portugal.}
\email{miguel.urbano@kaust.edu.sa}

\begin{document}

\begin{abstract}
We consider stationary $p$-Schr\"odinger equations on the whole space with integrable data and potentials that are confining in measure. We introduce asymptotic energy solutions in an asymptotic $L^p$ framework and establish existence and uniqueness in the degenerate range $p\ge2$. The proof relies on a new Rellich--Kondrachov-type compactness theorem of independent interest, which provides sufficient conditions for families of Sobolev functions to be precompact in asymptotic $L^p$ spaces, without any dimension-dependent restriction on the exponent. For data in the duality regime $L^1(\R^n)\cap L^{p'}(\R^n)$, asymptotic energy solutions coincide with weak energy solutions. We also show that additional compactness assumptions yield localized entropy-type solutions and, under suitable local regularity, distributional solutions.
\end{abstract}



\maketitle
\thispagestyle{empty}

\section{Introduction and main results}

Let $n\in\N$ and $1<p<\infty$. In this paper, we consider the equation
\begin{equation}\label{eq:p-Schrodinger}
-\Div\bigl(|\nabla u|^{p-2}\nabla u\bigr)+V|u|^{p-2}u=f
\qquad\text{in }\R^n,
\end{equation}
where $f\in L^1(\R^n)$ and $V\in L^\infty_{\rm loc}(\R^n)$ are given. We regard~\eqref{eq:p-Schrodinger} as a stationary $p$-Schr\"odinger-type equation on the whole space: the $p$-Laplacian plays the role of a nonlinear kinetic term, while the potential $V$ acts as an external field. In the linear case $p=2$, stationary whole-space Schr\"odinger equations obtained by the standing-wave ansatz form a classical topic; see, for instance,~\cite{strauss1977solitary, rabinowitz1992schrodinger}. In that literature, one typically studies stationary equations with right-hand side~$F(u)$ for some nonlinearity~$F$. Here, instead, we consider the equation with prescribed datum $f\in L^1(\R^n)$ under a confining assumption on the potential.

Since the domain is $\R^n$, confinement is produced by the potential rather
than by boundary conditions. In the theory of trapped Schr\"odinger equations,
a \emph{confining potential} is typically a potential such that $V(x) \to \infty$ as $|x| \to \infty$, so that mass is prevented from escaping to infinity. A standard
quantitative version of this requirement is the existence of constants
$\kappa,\gamma>0$ and $R_0\ge0$ such that
\[
V(x)\ge \kappa |x|^\gamma
\qquad\text{for a.e. }x\in\R^n\text{ with }|x|\ge R_0.
\]
In this paper, we work with a weaker notion. We say that $V$ is
\emph{confining in measure} if there exist $\kappa,\gamma>0$ such that the sets
\begin{equation}\label{eq:set_E_R}
E_R:=\bigl\{x\in\R^n:\ |x|\geq R,\ V(x)<\kappa |x|^\gamma\bigr\}, \qquad R>0,
\end{equation}
satisfy
\begin{equation}\label{eq:confinement_measure}
|E_R|\to 0
\qquad\text{as }R\to\infty.
\end{equation}
In other words, $V$ may fall below the confining profile $\kappa |x|^\gamma$
even arbitrarily far from the origin, but the set where this happens becomes
negligible at infinity. Here and throughout, $|A|$ denotes the Lebesgue measure of a subset $A \subseteq \R^n$.

Physically, condition~\eqref{eq:confinement_measure} models an external trap that is effective at large
distances except on a sparse family of defective regions where the potential is
locally weaker. This is consistent with whole-space trapped models from the
theory of confined quantum gases, and also with the presence of disorder or
impurities in trapping potentials; see, for instance,
\cite{dalfovo1999trapped,sanchez2010disordered,lye2005random,billy2008anderson}.
Our sparse-wells example makes this picture explicit: the background trap grows
like $|x|^\gamma$, but it contains infinitely many small wells placed farther
and farther away, with finite total measure; see Example~\ref{ex:sparse-wells}.

Our main goal is to develop an $L^1$-theory for \eqref{eq:p-Schrodinger} on
$\R^n$ under this weak confinement assumption. The natural solution space is not
the classical space $L^p(\R^n)$, but the asymptotic space
\begin{equation}\label{eq:spaceLambda}
\Lambda^p(\R^n)
\coloneqq
\left\{u:\R^n\to\R\text{ measurable} \ \Big| \
\int_{\R^n}\min(|u|,1)^p\,\dx<\infty\right\},
\end{equation}
introduced in~\cite{alves2025F}. Some basic facts on asymptotic $L^p$ spaces,
including their characterization in terms of almost-$L^p$ functions on general
measure spaces, are collected in Appendix~\ref{appendix}.

This space is intrinsic to truncation methods: it is defined through the
$p$-power of the basic truncation $\min(|\idot|,1)$. In particular, the quantity
$\|\!\min(|u|,1)\|_p$ measures the size of the set where $u$ is large, rather
than its full $L^p$-mass. The induced translation-invariant metric is
\[
d(u,v)\coloneqq \|\!\min(|u-v|,1)\|_p,
\]
and with this metric $\Lambda^p(\R^n)$ is a complete metrizable topological
vector space; in the terminology of Kalton, Peck and
Roberts~\cite{kalton1984sampler}, it is an $\mathrm{F}$-space. At the same time, $\Lambda^p(\R^n)$ is highly nonclassical: it is nonlocally convex, nonlocally bounded, and has
trivial dual; see~\cite{alves2025F}. In particular, weak convergence methods are unavailable in this setting. Nevertheless, one can still develop a well-posedness theory in these spaces. The present paper provides a first existence and uniqueness result in this asymptotic framework.

We introduce a notion of \emph{asymptotic energy solution} to
\eqref{eq:p-Schrodinger} defined through approximation in the energy space
$\X$ given by
\begin{equation}\label{eq:spaceX}
\X
:=
\Bigl\{v\in W^{1,p}(\R^n):\,
\int_{\R^n}V\,|v|^p\,\dx<\infty\Bigr\},
\end{equation}
and endowed with the norm
\begin{equation} \label{eq:normX}
\|v\|_{\X}
:=
\left(
\int_{\R^n}|\nabla v|^p\,\dx
+
\int_{\R^n}V\,|v|^p\,\dx
\right)^{1/p}.
\end{equation}
Roughly speaking, a function $u\in\Lambda^p(\R^n)$ is an asymptotic
energy solution if every truncation $T_\alpha(u)$ belongs to $\X$, and if $u$ arises as the limit of weak energy solutions to approximate problems, with
convergence of each truncation in the norm of $\X$; see Definition~\ref{def:ATS}. Here and throughout, for each $t>0$, we denote by $T_t$ the truncation function
\begin{equation}\label{eq:truncations}
T_t(s)=\max\big\{\!-t,\min\{s,t\}\big\},
\qquad s\in\R.
\end{equation}

Our first main result is the following existence and uniqueness theorem.

\begin{theorem}\label{thm:AES}
Assume $p\ge2$, let $f\in L^1(\R^n)$, and let
$V\in L^\infty_{\rm loc}(\R^n)$ satisfy $V\ge1$ a.e.\ and be confining in
measure. Then there exists a unique asymptotic energy solution
$u\in \Lambda^p(\R^n)$ of~\eqref{eq:p-Schrodinger}. Moreover, for every
$\alpha>0$,
\begin{equation} \label{eq:AES}
\|T_\alpha(u)\|_{\X}^p
\le \alpha\,\|f\|_1.
\end{equation}
\end{theorem}

As usual, we denote by $p'=p/(p-1)$ the conjugate exponent of $p$. The proof of
Theorem~\ref{thm:AES} is based on an approximation procedure in the energy space
$\X$. One solves regularized problems for data
$f_j\in L^1(\R^n)\cap L^{p'}(\R^n)$, derives uniform estimates for the
truncations of the corresponding weak energy solutions, and then passes to the
limit in the asymptotic space $\Lambda^p(\R^n)$. In this way, one obtains a
canonical limit object in $\Lambda^p(\R^n)$, characterized by the convergence
of all truncations in the norm of~$\X$.

In the duality regime $f\in L^1(\R^n)\cap L^{p'}(\R^n)$, the weak energy solution is itself an asymptotic energy solution, so the new
notion is consistent with the classical variational formulation. For general $f\in L^1(\R^n)$, the asymptotic energy solution is the natural notion of solution yielded by the present compactness method. We also show that,
under additional information on the approximating sequence, this solution can be
upgraded to a localized entropy-type formulation and, under further local
regularity, to the usual distributional equation.

The restriction $p\ge2$ comes from the stability estimate for the approximating
problems. In this range, the monotonicity of the $p$-Laplace vector field yields
direct control of the energy norm of truncations of the difference of two
solutions. This estimate is used both in the existence proof, where one shows
that the compactness limit is an asymptotic energy solution, and in the
uniqueness argument. For $1<p<2$, the corresponding monotonicity estimate is
weaker and does not lead, through the same argument, to convergence in the
energy norm.

A key ingredient in the proof of Theorem~\ref{thm:AES} is a compactness principle in the space~$\Lambda^p(\R^n)$ which is of independent interest. This compactness result is the mechanism that produces the asymptotic solution concept and, in our view, one of the main contributions of the paper. It gives a Rellich--Kondrachov-type criterion for total boundedness in $\Lambda^p(\R^n)$ under weak Sobolev control and a truncated decay condition at infinity. This furnishes a new structural tool for the analysis of PDEs in asymptotic spaces, with possible applications to other whole-space problems with very weak data.

 For $1\le q<\infty$, we denote by $L^{q,\infty}(\R^n)$ the usual weak $L^q$
space, endowed with the quasi-norm
\[
\|h\|_{q,\infty}
:=
\sup_{\lambda>0}\lambda\,\big|\{x\in\R^n:\ |h(x)|>\lambda\}\big|^{1/q}.
\]
Unless otherwise specified, $L^p$-norms are taken over $\R^n$, and we write
\[
\|\idot\|_p:=\|\idot\|_{L^p(\R^n)}.
\]

Our second main result is the following Rellich--Kondrachov-type compactness theorem.
\begin{theorem}\label{thm_ARK}
Let $1\le p<\infty$, and let $\mathcal F$ be a family of functions in
$W^{1,p}(\R^n)$. Assume that the following conditions hold:
\begin{enumerate}[(i)]
\item There exist $C>0$ and $1<q<\infty$ such that
\begin{equation} \label{eq:ARK1}
\|f\|_p+\|\nabla f\|_{q,\infty}\le C
\end{equation}
for all $f\in\mathcal F$.
\smallskip
\item For every $\varepsilon>0$, there exists $R>0$ such that
\begin{equation}\label{eq:ARK2}
\int_{|x|>R}\min(|f(x)|,1)^p\,\dx<\varepsilon^p
\end{equation}
for all $f\in\mathcal F$.
\end{enumerate}
Then $\mathcal F$ is totally bounded in $\Lambda^p(\R^n)$.
\end{theorem}

A key point is that Theorem~\ref{thm_ARK} imposes \emph{no restriction relating
$p$ to the dimension}. This is in sharp contrast with the classical
Rellich--Kondrachov theorem, where compact embedding into Lebesgue spaces is
governed by the critical Sobolev exponent; see, for instance,
\cite{evans2010partial,hanche2010kolmogorov,hanche2016addendum}. Here, however, the target space is not a Lebesgue space but the asymptotic space $\Lambda^p(\R^n)$, whose topology is substantially weaker. The proof builds on the Kolmogorov--Riesz compactness theorem for $\Lambda^p(\R^n)$ obtained in~\cite{alves2026kolmogorov} and combines it with translation estimates for Sobolev functions derived through the Hardy--Littlewood maximal operator; see
also~\cite{nguyen2021levelset} for related level-set and maximal-function techniques.

The well-posedness result in Theorem~\ref{thm:AES} should also be viewed in the context of the $L^1$ theory and of the treatment of measure data for nonlinear
elliptic equations. On
bounded domains, truncation methods led to the foundational $L^1$-theory of entropy solutions by
B\'enilan et al.~\cite{benilan1995L1}. Closely related developments include renormalized formulations, as well as equations with measure data and absorption potentials; see, for instance,
\cite{boccardo1989measure, boccardo1996entropy, dalmaso1999renormalized, saintier2021absorption, gkikas2024absorption}.
Approximation-based notions of solution were also studied
in~\cite{dalmaso1997reachable}, and the entropy theory has been extended to the
variable-exponent setting in~\cite{sanchon2009entropy}.

On the whole space, nonlinear elliptic equations with rough data and no
prescribed boundary behavior were already studied by Boccardo, Gallou\"et and
V\'azquez~\cite{boccardo1993whole}. At first sight, their framework appears
more general, since it treats a broader class of operators and merely local
assumptions on the datum. However, the two theories address different regimes.
In the model equation
\[-\Div\bigl(|\nabla u|^{p-2}\nabla u\bigr)+|u|^{s-1}u=f,\]
the whole-space existence theory in~\cite{boccardo1993whole} is driven by an
absorption term of order $s$ strictly larger than $p-1$. By contrast, the
problem studied here has the borderline lower-order growth $V|u|^{p-2}u$, where
global compactness is recovered through confinement at infinity rather than
through stronger absorption in $u$. Moreover, the solutions
in~\cite{boccardo1993whole} are obtained in a local Sobolev framework, whereas
the present paper leads to a different type of whole-space theory: all
truncations of the solution belong to the weighted energy space $\X$, while the
solution itself is described by asymptotic $L^p$-integrability and is obtained
as the limit of an approximation scheme in $\X$.

The paper is organized as follows. In Section~\ref{sec:confinement}, we collect
some basic facts about confinement in measure and compare it with the usual pointwise growth condition. Section~\ref{sec:proof_ARK} is devoted to the proof of the asymptotic Rellich--Kondrachov theorem. In Section~\ref{sec:approximate}, we introduce the approximate problems and derive the basic stability estimates. In Section~\ref{sec:passage_limit}, we pass to the limit in the approximation scheme and derive some properties of the limit
function. Section~\ref{sec:asymptotic_ES} contains the proof of
Theorem~\ref{thm:AES} and the consistency of asymptotic energy solutions with
the weak formulation under higher integrability of the datum. Finally, in
Section~\ref{sec:cond_upgrades}, we study conditional upgrades to localized and
distributional formulations.

Throughout the paper, $\N$ denotes the set of positive integers. For local norms
we always indicate the domain explicitly, for instance $\|\cdot\|_{L^p(B)}$.
Given $x\in\R^n$ and $r>0$, the open ball centered at $x$ with radius $r$ is
denoted by $B(x,r)$. At the origin, we abbreviate this to~$B_r$.

\section{Confinement in measure} \label{sec:confinement}

We collect here some basic facts about the notion of confinement in measure
introduced above. We first show that condition~\eqref{eq:confinement_measure}
is equivalent to the global sublevel set
\[
\{x\in\R^n:\ V(x)<\kappa |x|^\gamma\}
\]
having finite measure. We then compare it with the usual pointwise growth condition
\[
V(x)\ge \kappa |x|^\gamma
\qquad\text{for }|x|\text{ sufficiently large},
\]
and finally present a sparse-wells example showing that confinement in measure
is strictly weaker.

\begin{lemma}\label{lem:conf-measure-finite-bad-set}
Let $V:\R^n\to\R$ be measurable, fix $\kappa,\gamma>0$, and consider the sets $E_R$ defined in~\eqref{eq:set_E_R}. Then~\eqref{eq:confinement_measure} holds if and only if
\begin{equation}\label{eq:finite-bad-set}
\big|\{x\in\R^n:\ V(x)<\kappa |x|^\gamma\}\big|<\infty.
\end{equation}
\end{lemma}

\begin{proof}
Set
\[
E:=\{x\in\R^n:\ V(x)<\kappa |x|^\gamma\}.
\]
Then
\[
E_R=E\cap\{|x|\geq R\}
\qquad\text{for every }R>0.
\]

Assume first that $|E|<\infty$. Since the sets $E_R$ decrease to the empty set
as $R\to\infty$, continuity from above yields
\[
|E_R|\to 0.
\]
Hence $V$ is confining in measure.

Conversely, assume that $|E_R|\to 0$ as $R\to\infty$. Choose $R_0>0$ such that
$|E_{R_0}|<1$. Then
\[
E=(E\cap B_{R_0})\cup E_{R_0}.
\]
Since $B_{R_0}$ has finite measure, also $E\cap B_{R_0}$ has finite measure,
and therefore $|E|<\infty$.
\end{proof}

\begin{lemma}\label{lem:classic-implies-measure}
Assume that there exist $\kappa,\gamma>0$ and $R_0\ge 0$ such that
\begin{equation}\label{eq:classic-confinement}
V(x)\ge \kappa |x|^\gamma
\qquad\text{for a.e. }x\in\R^n\text{ with }|x|\ge R_0.
\end{equation}
Then $V$ is confining in measure.
\end{lemma}

\begin{proof}
Simply note that if~\eqref{eq:classic-confinement} holds, then for every $R\ge R_0$ we have 
\[|E_R|= \big| \bigl\{x\in\R^n:\ |x|\geq R,\ V(x)<\kappa |x|^\gamma\bigr\} \big| = 0.\]
\end{proof}

\begin{example}[Sparse wells]\label{ex:sparse-wells}
Fix $n\ge 1$ and $\gamma>0$, and let $e_1=(1,0,\dots,0)\in\R^n$. For each
$k\in\N$ set
\[
x_k:=2^k e_1,
\qquad
r_k:=2^{-2k},
\qquad
U:=\bigcup_{k=1}^\infty B(x_k,r_k).
\]
Define $V:\R^n\to[1,\infty)$ by
\begin{equation}\label{eq:V-sparse-wells}
V(x):=
\begin{cases}
1, & x\in U,\\
1+|x|^\gamma, & x\notin U.
\end{cases}
\end{equation}
We prove that $V$ is confining in measure, but $V$ does \emph{not} satisfy the
classical confinement condition \eqref{eq:classic-confinement}.

\par 
We first note that
\[
U\subseteq \{|x|>1\}.
\]
Indeed, for every $k\ge 1$ and every $x\in B(x_k,r_k)$,
\[
|x|\ge |x_k|-|x-x_k|>2^k-2^{-2k}\ge 2-\tfrac14>1.
\]
\noindent
\emph{(i) $V$ is confining in measure.}
Outside $U$ we have $V(x)=1+|x|^\gamma>|x|^\gamma$, so
\[
\{x\in\R^n:\ V(x)<|x|^\gamma\}\subseteq U.
\]
Conversely, if $x\in U$, then $V(x)=1$ and, since $|x|>1$ on $U$, we have
$1<|x|^\gamma$. Hence
\[
U\subseteq \{x\in\R^n:\ V(x)<|x|^\gamma\}.
\]
Therefore
\[
\{x\in\R^n:\ V(x)<|x|^\gamma\}=U.
\]
Since
\[
|U|
\le \sum_{k=1}^\infty |B(x_k,r_k)|
= \omega_n \sum_{k=1}^\infty r_k^n
= \omega_n \sum_{k=1}^\infty 2^{-2kn}
<\infty,
\]
where $\omega_n$ denotes the measure of the unit ball in $\R^n$, Lemma~\ref{lem:conf-measure-finite-bad-set} shows that $V$ is confining in measure.
\medskip \\
\emph{(ii) $V$ does not satisfy~\eqref{eq:classic-confinement}.}
Suppose by contradiction that there exist $\kappa>0$, $\beta>0$, and $R_0\ge 0$
such that
\begin{equation} \label{eq:sparse_aux}
V(x)\ge \kappa |x|^\beta
\qquad\text{for a.e. }x\in\R^n\text{ with }|x|\ge R_0.
\end{equation}
Choose $k$ so large that
\[
|x_k|-r_k>R_0
\qquad\text{and}\qquad
\kappa (|x_k|-r_k)^\beta>1.
\]
Then, for $x \in B(x_k,r_k)$,
\[|x| \geq |x_k|-r_k > R_0 \qquad \text{and} \qquad  \kappa |x|^\beta\ge \kappa (|x_k|-r_k)^\beta>1.\]
This together with~\eqref{eq:sparse_aux} yields that
\[V(x) \geq \kappa |x|^\beta > 1 \]
for a.e.\ $x\in B(x_k,r_k)$.

 On the other hand, by the definition of $V$ in~\eqref{eq:V-sparse-wells},
\[
V(x)=1
\qquad\text{for all }x\in B(x_k,r_k),
\]
a contradiction.
\end{example}

\section{Proof of the asymptotic Rellich--Kondrachov theorem} \label{sec:proof_ARK}

We shall use the Kolmogorov--Riesz compactness theorem in $\Lambda^p(\R^n)$ proved in~\cite{alves2026kolmogorov}. Since $\Lambda^p(\R^n)$ is complete, total boundedness and relative compactness coincide. We recall the result for
the reader's convenience.
\begin{theorem}[Kolmogorov--Riesz compactness theorem in $\Lambda^p(\R^n)$ \cite{alves2026kolmogorov}] \label{thm_KR_Lambda} 
A subset $\mathcal{F} \subseteq \Lambda^p(\R^n)$, $1 \leq p < \infty$, is totally bounded in~$\Lambda^p(\R^n)$ if, and only if, the following three conditions hold:
\begin{enumerate}[(i)]
\item For each $\varepsilon > 0$, there exists $r > 0$ such that 
\[
\int_{\R^n} \min(|f(x+y) - f(x)|,1)^p \, \dx < \varepsilon^p,
\]
for every $y \in \R^n$ with $|y| < r$ and all $f \in \mathcal{F}$.
\smallskip
\item For each $\varepsilon > 0$, there exists $R > 0$ such that 
\[
\int_{|x|>R} \min(|f(x)|,1)^p \, \dx < \varepsilon^p,
\]
for all $f \in \mathcal{F}$.
\smallskip
\item For each $\varepsilon > 0$, there exists $K > 0$ such that 
\[
\big| \{ |f| > K \} \big| < \varepsilon,
\]
for all $f \in \mathcal{F}$.
\end{enumerate}
\end{theorem}

\subsection{Lemmas on the Hardy--Littlewood maximal operator}
The proof of Theorem~\ref{thm_ARK} relies on~Theorem~\ref{thm_KR_Lambda} together with standard maximal function estimates for Sobolev functions; cf.\ \cite{hajlasz1996metric}. We record the precise form needed below and include proofs for the reader's convenience.

For a ball $B\subseteq\R^n$, we denote by $f_B$ the average of a function $f$
over $B$,
\[
f_B:=\frac{1}{|B|}\int_B f(z)\,\dz.
\]
The centered Hardy--Littlewood maximal operator is defined, for a locally
integrable function $h$, by
\begin{equation}\label{HL_operator}
M(h)(x):=\sup_{r>0}\frac{1}{|B(x,r)|}\int_{B(x,r)}|h(z)|\,\dz.
\end{equation}
We recall that the Hardy--Littlewood maximal operator $M$ maps the space $L^{q,\infty}(\R^n)$ to itself for $1 < q < \infty$; see \cite[Exercise~2.1.13]{grafakos2014classical}.

\begin{lemma}
If $f \in W^{1,p}(\R^n)$, then for a.e.\ $x \in \R^n$,
\begin{equation} \label{eq_bound_M_aux}
|f(x) - f_{B(x,r)}| \leq C \, r \, M(|\nabla f|)(x)
\end{equation}
for some positive constant $C$ depending only on $n$.
\end{lemma}
\begin{proof}
Let $B_0 = B(x,r)$ and for each $k \in \N$ let $B_k = B(x, 2^{-k} r)$. For a Lebesgue point $x$ of $f$ we have $f_{B_k} \to f(x)$ as $k \to \infty$ and hence
\[|f(x) - f_{B_0}| = \left| \sum_{k=0}^\infty (f_{B_{k+1}} - f_{B_k}) \right| \leq \sum_{k=0}^\infty |f_{B_{k+1}} - f_{B_k}|.\] 
Since $B_{k+1} \subseteq B_k$ we have 
\begin{align*}
|f_{B_{k+1}} - f_{B_k}| & = \left|\frac{1}{|B_{k+1}|} \int_{B_{k+1}} f(z) - f_{B_k} \, \dz \right| \\
& \leq \frac{1}{|B_{k+1}|} \int_{B_{k+1}} |f(z) - f_{B_k}| \, \dz \\
& \leq \frac{2^n}{|B_{k}|} \int_{B_{k}} |f(z) - f_{B_k}| \, \dz.
\end{align*}
We now use the Poincar\'{e} inequality for a ball (see \cite{evans2010partial}) to deduce
\[\int_{B_{k}} |f(z) - f_{B_k}| \, \dz \leq C_n 2^{-k} r \int_{B_k} |\nabla f(z)| \, \dz \]
and hence
\begin{align*}
|f_{B_{k+1}} - f_{B_k}| &\leq C_n \frac{2^n}{|B_{k}|} 2^{-k} r \int_{B_k} |\nabla f(z)| \, \dz \\
&\leq C_n 2^n 2^{-k} r M(|\nabla f|)(x). 
\end{align*}
Consequently,
\begin{align*}
|f(x) - f_{B_0}| & \leq C_n 2^n r M(|\nabla f|)(x) \sum_{k=0}^\infty 2^{-k} \\
& = C_n 2^{n+1} r M(|\nabla f|)(x)
\end{align*}
which concludes the proof.
\end{proof}

\begin{lemma} \label{lem:bound_M}
If $f \in W^{1,p}(\R^n)$, then for every $y\in\R^n$,
\begin{equation} \label{eq_bound_M}
|f(x+y) - f(x)| \leq C \, |y| \, \big(M(|\nabla f|)(x+y) +  M(|\nabla f|)(x)\big)
\end{equation}
for a.e.\ $x \in \R^n$, where $C>0$ depends only on $n$.
\end{lemma}
\begin{proof}
Fix $y\in\R^n$. The case $y=0$ is trivial, so assume $y\neq 0$. Let $N\subseteq\R^n$
be the null set of points that are not Lebesgue points of $f$. Then, for every
\[
x\in \R^n\setminus \bigl(N\cup (N-y)\bigr),
\]
both $x$ and $x+y$ are Lebesgue points of $f$. Set
\[
z:=x+y,
\qquad
r:=|y|.
\]
Let $B = B(z,2r)$ and $\widehat{B} = B(x,3r)$. Note that $x \in B$ and
$B \subseteq \widehat{B}$. By the previous lemma, we have
\begin{align*}
|f(z) - f(x)|&  \leq |f(z) - f_B| + |f_B - f(x)| \\
& \leq C 2r M(|\nabla f|)(z) + |f(x) - f_{\widehat{B}}| + |f_{\widehat{B}} -f_B| \\
& \leq C 2r M(|\nabla f|)(z) + C 3r M(|\nabla f|)(x) + |f_{\widehat{B}} -f_B|.
\end{align*}
We now estimate the last term on the right-hand side using the Poincar\'e
inequality for a ball:
\begin{align*}
|f_{\widehat{B}} -f_B|
& = \left|\frac{1}{|B|} \int_B (f_{\widehat{B}} - f(w)) \, \rd w  \right|  \\
& \leq \frac{1}{|B|} \int_B |f_{\widehat{B}} - f(w)| \, \rd w \\
& \leq \frac{1}{|B|} \int_{\widehat{B}} |f_{\widehat{B}} - f(w)| \, \rd w \\
& \leq \frac{C_n}{|B|} 3r \int_{\widehat{B}} |\nabla f(w)| \, \rd w \\
& = C_n \frac{3^{n+1}}{2^n} r \frac{1}{|\widehat{B}|} \int_{\widehat{B}} |\nabla f(w)| \, \rd w \\
& \leq C_n \frac{3^{n+1}}{2^n} r M(|\nabla f|)(x).
\end{align*}
Combining the previous estimates and recalling that $z=x+y$ and $r=|y|$, we obtain
\[
|f(x+y)-f(x)|
\leq C\,|y|\big(M(|\nabla f|)(x+y)+M(|\nabla f|)(x)\big)
\]
for every $x\in \R^n\setminus \bigl(N\cup (N-y)\bigr)$. Since
$N\cup (N-y)$ has measure zero, the proof is complete.
\end{proof}

\subsection{Proof of Theorem~\ref{thm_ARK}}

We show that $\mathcal{F}$ satisfies the conditions of Theorem~\ref{thm_KR_Lambda}. Fix $\eps > 0$ and $f \in \mathcal{F}$.

 We first observe that for $K > C/\eps^{1/p}$, using the $L^p$-bound in~\eqref{eq:ARK1} we have 
\[\big| \{|f| > K \} \big| = \frac{1}{K^p} \int_{|f| > K} K^p \, \dx \leq \frac{1}{K^p} \int_{\R^n} |f|^p \, \dx \leq \frac{C^p}{K^p} < \eps, \]
so $\mathcal{F}$ satisfies the third condition of Theorem~\ref{thm_KR_Lambda}.

Given that the second condition of~Theorem~\ref{thm_KR_Lambda} is already part of the current assumptions, it remains to prove that the first condition holds, that is, we need to find $r = r(\eps)>0$ so that
\[\int_{\R^n} \min(|f(x+y) - f(x)|,1)^p \, \dx < \eps^p, \]
for $|y| < r$. By hypothesis, one can pick $R=R(\eps) > 0$ such that 
\begin{equation} \label{eq:thproof_aux1}
\int_{|x| > R/2} \min(|f(x)|,1)^p \, \dx < \frac{\eps^p}{2^{p+1}}.
\end{equation}
Using this $R$, we split the integral 
\begin{align*}
\int_{\R^n} &\min(|f(x+y) - f(x)|,1)^p \, \dx \\
 = \ & \int_{|x|>R} \min(|f(x+y) - f(x)|,1)^p \, \dx  \\
 & + \int_{|x|\leq R} \min(|f(x+y) - f(x)|,1)^p \, \dx  \\
 \leq \ & 2^{p-1} \int_{|x|>R} \min(|f(x+y)|,1)^p \, \dx + 2^{p-1} \int_{|x|>R} \min(|f(x)|,1)^p \, \dx\\
 & + \int_{|x|\leq R} \min(|f(x+y) - f(x)|,1)^p \, \dx
\end{align*}
where in the last step we used the inequalities $\min(|a+b|,1) \leq \min(|a|,1) + \min(|b|,1)$ and $(|a|+|b|)^p \leq 2^{p-1}(|a|^p+|b|^p)$.\par 
Using \eqref{eq:thproof_aux1} we deduce
\[\int_{|x|>R} \min(|f(x)|,1)^p \, \dx \leq \int_{|x|>R/2} \min(|f(x)|,1)^p \, \dx < \frac{\eps^p}{2^{p+1}}, \]
and, for $|y| < R/2$, the change of variables $x+y=z$ gives
\[\int_{|x|>R} \min(|f(x+y)|,1)^p \, \dx \leq \int_{|z|>R/2} \min(|f(z)|,1)^p \, \dz < \frac{\eps^p}{2^{p+1}}. \]
Therefore, for $|y| < R/2$ we have
\begin{equation} \label{eq:thproof_aux2}
\int_{\R^n} \min(|f(x+y) - f(x)|,1)^p \, \dx < \frac{\eps^p}{2} + \int_{|x|\leq R} \min(|f(x+y) - f(x)|,1)^p \, \dx.
\end{equation}
\par
For $\lambda > 0$ and $y \in \R^n$, we consider the sets $E_\lambda$ and $S_\lambda(y)$ given by
\[E_\lambda = \big\{x \in \R^n :\, M(|\nabla f|)(x) > \lambda \big\}, \qquad S_\lambda(y) = E_\lambda \cup (E_\lambda - y). \]
Observe that 
\begin{equation} \label{eq:proof_RK_aux1}
\begin{split}
|S_\lambda(y)| & \leq 2 |E_\lambda| \\
& \leq 2\lambda^{-q} \|M(|\nabla f|) \|_{q,\infty}^q \\
&\leq C_{n,q} \lambda^{-q} \|\nabla f\|_{q,\infty}^q \\
& \leq C_1 \lambda^{-q}, 
\end{split}
\end{equation}
for some positive constant $C_1$ depending only on $n$ and $q$, where we used the fact that $M$ maps $L^{q,\infty}(\R^n)$ to itself and the hypothesis that the gradients of functions in $\mathcal{F}$ are uniformly bounded in $L^{q,\infty}(\R^n)$.

Moreover, if $x \in S_\lambda(y)^c$, then $x+y,x \in E_\lambda^c$ and hence,
for a.e.\ $x\in S_\lambda(y)^c$, Lemma~\ref{lem:bound_M} gives
\begin{equation} \label{eq:proof_RK_aux2}
\begin{split}
|f(x+y) - f(x)| & \leq C_n |y| \big(M(|\nabla f|)(x+y) + M(|\nabla f|)(x) \big) \\
& \leq 2 C_n \lambda |y|.
\end{split}
\end{equation}

From~\eqref{eq:proof_RK_aux1}--\eqref{eq:proof_RK_aux2} we deduce
\begin{align*}
\int_{|x|\leq R} & \min(|f(x+y) - f(x)|,1)^p \, \dx \\
 = \ &  \int_{B_R \cap S_\lambda(y)}  \min(|f(x+y) - f(x)|,1)^p \, \dx \\ & + \int_{B_R \cap S_\lambda^c(y)}  \min(|f(x+y) - f(x)|,1)^p \, \dx \\
  \leq \ &  |S_\lambda(y)| + \int_{B_R \cap S_\lambda^c(y)}  |f(x+y) - f(x)|^p \, \dx \\
  \leq \ & C_1 \lambda^{-q} + C_2 R^n \lambda^p |y|^p, 
\end{align*}
for some positive constant $C_2$ depending only on $n$ and $p$. Now, we first choose $\lambda$ large enough so that $C_1 \lambda^{-q} < \eps^p/4$, and then we choose $r_0$ small enough to guarantee that $C_2 R^n \lambda^p r_0^p < \eps^p/4$. Thus, for $r = \min\{r_0, R/2\}$ and $|y| < r$ it follows that
\[\int_{\R^n} \min(|f(x+y) - f(x)|,1)^p \, \dx < \frac{\eps^p}{2} + C_1 \lambda^{-q} + C_2 R^n \lambda^p r^p < \eps^p, \]
which concludes the proof. \qed

\section{Approximate solutions} \label{sec:approximate}
Throughout the next sections, we assume that $V\in L^\infty_{\rm loc}(\R^n)$ and $V(x)\ge1$ for a.e.\ $x\in\R^n$. The latter condition could be replaced by a lower bound by some fixed constant $\nu_0>0$; for simplicity, we take $\nu_0=1$.

Let $\{f_k\}_{k \in \N}\subseteq L^1(\R^n)\cap L^{p'}(\R^n)$ be such that
\[
\sup_{k \in \N}\|f_k\|_1\le \|f\|_1
\]
and
\[
f_k\to f\qquad\text{in }L^1(\R^n) \qquad \text{as } k \to \infty.
\]
A canonical choice is
\[
f_k=T_k(f)\chi_{B_k}.
\]
\par
Recall the energy space $\X$ introduced in \eqref{eq:spaceX}, endowed with the
norm \eqref{eq:normX}. Since $V\ge 1$ a.e., we have
\[
  \|v\|_{p}^p
  \le \int_{\R^n} V|v|^p\,\dx
  \le \|v\|_{\X}^p,
\]
and therefore
\[
  \|v\|_{W^{1,p}(\R^n)}^p
  := \int_{\R^n} |\nabla v|^p\,\dx + \int_{\R^n}|v|^p\,\dx
  \le \|v\|_{\X}^p.
\]
In particular, $\X$ is continuously embedded into $W^{1,p}(\R^n)$ and into $L^p(\R^n)$.

\subsection{Weak compactness in $\X$}

\begin{lemma} \label{lem:weakcompactX}
Suppose $\{v_k\}_{k \in \N}$ is bounded in $\X$. Then, there exist a subsequence $\{v_{k_j} \}_{j \in \N}$ and a function $v \in \X$ such that 
\[
v_{k_j} \rightharpoonup v \qquad\text{weakly in }L^p(\R^n) \qquad \text{as } j \to \infty,
\]
and
\[\| v\|_{\X} \leq \liminf_{j \to \infty} \| v_{k_j}\|_{\X}.\]
\end{lemma}
\begin{proof}
Let $\{v_k\}_{k \in \N}$ be bounded in $\X$. Then, the sequences $\{v_k\}_{k \in \N}$, $\{ V^{1/p} \, v_k\}_{k \in \N}$ are bounded in $L^p(\R^n)$, and $\{\nabla v_k \}_{k \in \N}$ is bounded in $L^p(\R^n;\R^n)$. By reflexivity of $L^p$ spaces for $1 < p < \infty$,
after passing to a subsequence $\{v_{k_j} \}_{j \in \N}$ we may assume that there exist
\[
v\in L^p(\R^n),\qquad G\in L^p(\R^n;\R^n),\qquad z\in L^p(\R^n),
\]
such that
\[
v_{k_j} \rightharpoonup v \qquad\text{weakly in }L^p(\R^n),
\]
\[
\nabla v_{k_j} \rightharpoonup G \qquad\text{weakly in }L^p(\R^n;\R^n),
\]
\[
V^{1/p}\,v_{k_j} \rightharpoonup z \qquad\text{weakly in }L^p(\R^n).
\]
We claim that $v\in \X$, $\nabla v=G$, and $z=V^{1/p}\,v$. Indeed, for $\phi\in C_c^\infty(\R^n)$ and each $i\in\{1,\dots,n\}$,
\[
\int_{\R^n} v_{k_j}\,\partial_i\phi\,\dx
=
-\int_{\R^n} \partial_i v_{k_j}\,\phi\,\dx.
\]
Passing to the limit yields
\[
\int_{\R^n} v\,\partial_i\phi\,\dx
=
-\int_{\R^n} G_i\,\phi\,\dx,
\]
so $v\in W^{1,p}(\R^n)$ with $\partial_i v=G_i$. Next, since $V\in L^\infty_{\loc}(\R^n)$ we have
$V^{1/p} \, \phi\in L^{p'}(\R^n)$, and therefore
\begin{align*}
\int_{\R^n} z\,\phi\,\dx & =
\lim_{j\to\infty}\int_{\R^n} V^{1/p}\,v_{k_j}\,\phi\,\dx \\
& =
\lim_{j\to\infty}\int_{\R^n} v_{k_j}\,(V^{1/p} \, \phi)\,\dx \\
& =
\int_{\R^n} v\,V^{1/p}\phi\,\dx.
\end{align*}
Hence $z=V^{1/p} \, v$ a.e.\ in $\R^n$. Since
$z\in L^p(\R^n)$, it follows that
\[
\int_{\R^n}V \,|v|^p\,\dx
=
\int_{\R^n}|z|^p\,\dx<\infty.
\]
Thus $v\in \X$.

By weak lower semicontinuity in $L^p$,
\[
\int_{\R^n}|\nabla v|^p\,\dx
\le
\liminf_{j\to\infty}\int_{\R^n}|\nabla v_{k_j}|^p\,\dx,
\]
and
\begin{align*}
\int_{\R^n}V \, |v|^p\,\dx
& =
\int_{\R^n}|V^{1/p} \, v|^p\,\dx \\
& \le
\liminf_{j\to\infty}\int_{\R^n}|V^{1/p} \, v_{k_j}|^p\,\dx \\
& =
\liminf_{j\to\infty}\int_{\R^n}V \, |v_{k_j}|^p\,\dx,
\end{align*}
from which it follows that 
\[\| v\|_{\X} \leq \liminf_{j \to \infty} \| v_{k_j}\|_{\X}.\]
\end{proof}

\subsection{Weak energy solutions}

The next lemma provides existence and uniqueness of \emph{weak energy solutions} to~\eqref{eq:p-Schrodinger} with datum~$f_k \in L^{p'}(\R^n)$.
\begin{lemma}\label{lem:approx_uk}
For each $k\in\N$ there exists a unique $u_k\in \X$ such that
\begin{equation}\label{eq:exist_approx_uk}
  \int_{\R^n} |\nabla u_k|^{p-2} \nabla u_k\cdot \nabla\varphi\,\dx
  + \int_{\R^n} V\,|u_k|^{p-2}u_k\,\varphi\,\dx
  = \int_{\R^n} f_k\,\varphi\,\dx,
\end{equation}
for all $\varphi \in \X$. In particular, since $C_c^\infty(\R^n)\subseteq \X$, the function $u_k$ is a
distributional solution of~\eqref{eq:p-Schrodinger} with right-hand side $f_k$.
\end{lemma}

\begin{proof}
Define
\begin{equation} \label{eq:functional_Jk}
  J_k(v)
  =\frac1p\int_{\R^n}|\nabla v|^p\,\dx
   +\frac1p\int_{\R^n}V\,|v|^p\,\dx
   -\int_{\R^n} f_k \, v\,\dx,
  \qquad v\in \X.
\end{equation}
This functional is well-defined on $\X$: the first term is finite since
$v\in W^{1,p}(\R^n)$, the second term is finite by definition of $\X$, and the last
term is finite by H\"older's inequality since $f_k\in L^{p'}(\R^n)$ and
$v\in L^p(\R^n)$.

Moreover, by H\"older's inequality and the estimate $\|v\|_{p}\le \|V^{1/p} \,v \|_p$,
\[
  \left|\int_{\R^n} f_k v\,\dx\right|
  \le \|f_k\|_{p'}\|v\|_{p}
  \le \|f_k\|_{p'}
      \left(\int_{\R^n}V\,|v|^p\,\dx\right)^{1/p}.
\]
Applying Young's inequality gives
\[
  \left|\int_{\R^n} f_k v\,\dx\right|
  \le \frac{1}{2p}\int_{\R^n}V\,|v|^p\,\dx
      + C_p\|f_k\|_{p'}^{p'}
\]
for some constant $C_p>0$ depending only on $p$. Hence
\begin{align*}
 J_k(v) & \ge \frac1p\int_{\R^n}|\nabla v|^p\,\dx
     +\frac1{2p}\int_{\R^n}V \,|v|^p\,\dx
     -C_p\|f_k\|_{p'}^{p'} \\
     & \geq \frac1{2p} \|v \|_{\X}^p -C_p\|f_k\|_{p'}^{p'} \\ 
     & \geq -C_p\|f_k\|_{p'}^{p'}.
\end{align*}
Therefore, $J_k$ is coercive and bounded from below on $\X$. In particular, \[m \coloneqq \inf_{\X} J_k \in \R. \] 
By definition of infimum, for each $j \in \N$ there exists some $v_j \in \X$ such that 
\[m \leq J_k(v_j) < m + \frac1j \leq m + 1. \]
The sequence $\{v_j\}_{j\in\N}$ is a minimizing sequence for $J_k$. By coercivity, it follows that
\[\|v_j \|_{\X}^p \leq 2p\big(m + 1 + C_p\|f_k\|_{p'}^{p'}\big), \]
so $\{v_j\}_{j \in \N}$ is bounded in $\X$. By Lemma~\ref{lem:weakcompactX}, there exist a subsequence, still denoted by $\{v_j\}_{j\in\N}$, and a function $u_k \in \X$ such that 
\[
v_j \rightharpoonup u_k \qquad\text{weakly in }L^p(\R^n) \qquad \text{as } j \to \infty,
\]
and
\[\| u_k\|_{\X} \leq \liminf_{j \to \infty} \| v_j\|_{\X}.\]
Moreover, since $f_k\in L^{p'}(\R^n)$ and $v_j\rightharpoonup u_k$ weakly in $L^p(\R^n)$ as $j \to \infty$,
\[
\int_{\R^n} f_k \,v_j\,\dx \to \int_{\R^n} f_k \, u_k\,\dx.
\]
Therefore
\begin{align*}
J_k(u_k)  & = \frac{1}{p} \|u_k \|_{\X}^p - \int_{\R^n} f_k \, u_k \, \dx \\
& \leq \liminf_{j \to \infty} \frac{1}{p} \|v_j \|_{\X}^p - \limsup_{j \to \infty} \int_{\R^n} f_k \, u_k \, \dx \\
& \leq \liminf_{j\to\infty} J_k(v_j) \\
&=\inf_{\X} J_k,
\end{align*}
and since $u_k\in \X$, we conclude that $J_k(u_k)=\inf_{\X} J_k$.

We now derive the Euler--Lagrange equation. Fix $\varphi\in \X$ and define
\[
g(t):=J_k(u_k+t\varphi),\qquad t\in\R.
\]
Since $u_k$ minimizes $J_k$ on $\X$, the function $g$ has a minimum at $t=0$. We claim that $g$ is differentiable at $t=0$. Write
\[
g(t)=A(t)+B(t)-C(t),
\]
where
\[
A(t)=\frac1p\int_{\R^n}|\nabla u_k+t\nabla\varphi|^p\,\dx,
\qquad
B(t)=\frac1p\int_{\R^n}V \, |u_k+t\varphi|^p\,\dx,
\]
and
\[
C(t)=\int_{\R^n}f_k\,(u_k+t\varphi)\,\dx.
\]
For the first term, using the $C^1$ function $\xi \mapsto |\xi|^p/p$ on $\R^n$, by the fundamental theorem of calculus we get
\[
\frac{A(t)-A(0)}{t}
=
\int_{\R^n}\int_0^1
|\nabla u_k+s t\nabla\varphi|^{p-2}
(\nabla u_k+s t\nabla\varphi)\cdot\nabla\varphi\,\ds\,\dx.
\]
As $t\to0$, the integrand converges pointwise to
$|\nabla u_k|^{p-2}\nabla u_k\cdot\nabla\varphi$. Moreover, for $|t|\le1$,
\[
\bigl|
|\nabla u_k+s t\nabla\varphi|^{p-2}
(\nabla u_k+s t\nabla\varphi)\cdot\nabla\varphi
\bigr|
\le
C_p'\Bigl(|\nabla u_k|^{p-1}|\nabla\varphi|+|\nabla\varphi|^p\Bigr).
\]
Viewed as a function of $(x,s)\in \R^n\times(0,1)$ independent of $s$, the right-hand side
belongs to $L^1(\R^n\times(0,1))$ by H\"older's inequality. Hence, by dominated convergence,
\[
A'(0)=\int_{\R^n}|\nabla u_k|^{p-2}\nabla u_k\cdot\nabla\varphi\,\dx.
\]
Similarly, using the $C^1$ function $s \mapsto |s|^p/p$ on $\R$, we obtain
\[
\frac{B(t)-B(0)}{t}
=
\int_{\R^n}\int_0^1
V\,|u_k+s t\varphi|^{p-2}\,(u_k+s t\varphi) \, \varphi\,\ds\,\dx,
\]
and for $|t|\le1$,
\[
V \, |u_k+s t\varphi|^{p-1} \, |\varphi|
\le
C_p'\,V \Bigl(|u_k|^{p-1}|\varphi|+|\varphi|^p\Bigr).
\]
The right-hand side belongs to $L^1(\R^n)$ because
\begin{align*}
\int_{\R^n}V\,|u_k|^{p-1} \,|\varphi|\,\dx
& = \int_{\R^n}(V^{1/p}\,|u_k|)^{p-1} \,(V^{1/p} |\varphi|)\,\dx \\
& \le
\|V^{1/p}u_k\|_p^{p-1}\|V^{1/p}\varphi\|_p
<\infty,
\end{align*}
and $\int_{\R^n}V|\varphi|^p\,\dx<\infty$. Therefore,
\[
B'(0)=\int_{\R^n}V \, |u_k|^{p-2}\,u_k\,\varphi\,\dx.
\]
Finally,
\[
C'(0)=\int_{\R^n}f_k\,\varphi\,\dx.
\]
Hence
\[
g'(0)
=
\int_{\R^n}|\nabla u_k|^{p-2}\nabla u_k\cdot\nabla\varphi\,\dx
+\int_{\R^n}V \, |u_k|^{p-2}\,u_k\,\varphi\,\dx
-\int_{\R^n}f_k \,\varphi\,\dx.
\]
Since $g$ has a minimum at $0$, we have $g'(0)=0$ and thus
\[
\int_{\R^n}|\nabla u_k|^{p-2}\nabla u_k\cdot\nabla\varphi\,\dx
+\int_{\R^n}V(x)|u_k|^{p-2}u_k\,\varphi\,\dx
=
\int_{\R^n}f_k \, \varphi\,\dx.
\]

To prove uniqueness, let $u_k,w_k\in X$ be two solutions of
\eqref{eq:exist_approx_uk}. Subtracting the two weak formulations and testing with
$\varphi=u_k-w_k\in X$, we obtain
\begin{align*}
\int_{\R^n}
&\Big(|\nabla u_k|^{p-2}\nabla u_k-|\nabla w_k|^{p-2}\nabla w_k\Big)
\cdot(\nabla u_k-\nabla w_k)\,\dx
\\
&\quad
+\int_{\R^n}
V\Big(|u_k|^{p-2}u_k-|w_k|^{p-2}w_k\Big)(u_k-w_k)\,\dx
=0.
\end{align*}
Since the maps $\xi\mapsto |\xi|^{p-2}\,\xi$ on $\R^n$ and
$s\mapsto |s|^{p-2}\,s$ on $\R$ are monotone, both integrands are a.e.\ nonnegative.
Therefore, both integrals must vanish. In particular,
\[
\int_{\R^n}
V\Big(|u_k|^{p-2}u_k-|w_k|^{p-2}w_k\Big)(u_k-w_k)\,\dx=0.
\]
Since $V\ge 1$ a.e., we have
\[
0\le\Big(|u_k|^{p-2}u_k-|w_k|^{p-2}w_k\Big)(u_k-w_k) 
\le
 V\Big(|u_k|^{p-2}u_k-|w_k|^{p-2}w_k\Big)(u_k-w_k)
\]
for a.e.\ $x\in\R^n$. Hence
\[
\int_{\R^n}
\Big(|u_k|^{p-2}u_k-|w_k|^{p-2}w_k\Big)(u_k-w_k)\,\dx=0,
\]
and since the integrand is nonnegative, it follows that
\[
\Big(|u_k(x)|^{p-2}u_k(x)-|w_k(x)|^{p-2}w_k(x)\Big)(u_k(x)-w_k(x))=0
\quad\text{for a.e.\ }x\in\R^n.
\]
As $s\mapsto |s|^{p-2}s$ is strictly increasing on $\R$, we conclude that
$u_k=w_k$ a.e.\ in $\R^n$.
\end{proof}

\subsection{Truncation estimates}
For $t > 0$ we consider the truncations $T_t(u_k)$, where $u_k \in \X$ solves~\eqref{eq:exist_approx_uk}. Since $T_t:\R\to\R$ is Lipschitz, we have
$T_t(u_k)\in W^{1,p}(\R^n)$ with
\[
\nabla T_t(u_k)=T_t'(u_k)\nabla u_k \qquad\text{a.e. in }\R^n;
\]
see \cite{leoni2007chainrule}. Moreover, $|T_t(u_k)|\le |u_k|$ and hence
\[
\int_{\R^n}V \, |T_t(u_k)|^p\,\dx
\le
\int_{\R^n}V \, |u_k|^p\,\dx<\infty.
\]
Therefore $T_t(u_k)\in \X$ and may be used as a test function in
\eqref{eq:exist_approx_uk}.

\begin{lemma}\label{lem:energy_estimate_Tt(uk)}
Let $\{u_k\}_{k \in \N}$ be the sequence of approximate solutions from Lemma~\ref{lem:approx_uk}. Then, for every $t>0$ and every $k\in\N$,
\begin{equation}\label{eq:energy_estimate_Tt(uk)}
  \| T_t(u_k) \|_{\X}^p = \int_{\R^n} |\nabla T_t(u_k)|^p\,\dx
  + \int_{\R^n} V\,|T_t(u_k)|^p\,\dx
  \le t\,\|f\|_1.
\end{equation}
\end{lemma}

\begin{proof}
Fix $t>0$ and $k\in\N$. Taking $\varphi=T_t(u_k)$ as a test function in \eqref{eq:exist_approx_uk} we have
\[
\int_{\R^n} |\nabla u_k|^{p-2}\nabla u_k\cdot \nabla T_t(u_k)\,\dx
+\int_{\R^n} V\,|u_k|^{p-2}u_k\,T_t(u_k)\,\dx
=\int_{\R^n} f_k\,T_t(u_k)\, \dx.
\]
For the gradient term, we note that
\begin{align*}
\int_{\R^n} |\nabla u_k|^{p-2}\nabla u_k\cdot \nabla T_t(u_k)\,dx & = \int_{\R^n}|\nabla u_k|^{p}\chi_{\{|u_k|<t\}} \, \dx \\
& =\int_{\R^n} |\nabla T_t(u_k)|^{p}\,\dx.
\end{align*}
For the second term, pointwise, we have
\[
|u_k|^{p-2}u_k\,T_t(u_k)\ge |T_t(u_k)|^p.
\]
Indeed, if $|u_k|\le t$ then $T_t(u_k)=u_k$ and the two sides coincide; if $|u_k|>t$ then
$T_t(u_k)=t\,\mathrm{sign}(u_k)$ and so
\[
|u_k|^{p-2}u_k\,T_t(u_k)=t\,|u_k|^{p-1} > t^p=|T_t(u_k)|^p.
\]
Since $V$ is nonnegative, it follows that
\[
\int_{\R^n} V\,|u_k|^{p-2}u_k\,T_t(u_k)\,\dx
\ge \int_{\R^n} V\,|T_t(u_k)|^p\,\dx.
\]
Putting these into the weak formulation yields
\begin{equation*}
\int_{\R^n} |\nabla T_t(u_k)|^p\,\dx + \int_{\R^n} V\,|T_t(u_k)|^p\,\dx
\le \int_{\R^n} f_k\,T_t(u_k)\,\dx.
\end{equation*}
Finally, since $|T_t(u_k)|\le t$,
\[
\int_{\R^n} f_k\,T_t(u_k)\,\dx
\le \int_{\R^n} |f_k|\,|T_t(u_k)|\,\dx
\le t\,\|f_k\|_1
\le t\,\|f\|_1,
\]
which together with the previous estimate gives~\eqref{eq:energy_estimate_Tt(uk)}.
\end{proof}

\begin{lemma}\label{lem:tail-measure}
Let $\{u_k\}_{k \in \N}$ be the sequence of approximate solutions from Lemma~\ref{lem:approx_uk}. Then, for every $t > 0$, every $R>0$, and every $k\in\N$,
\begin{equation}\label{eq:tail-min-Tt(uk)}
\int_{|x|>R} \min(|T_t(u_k)|,1)^p\,\dx
\le |E_R| + \frac{t\|f\|_{1}}{\kappa R^\gamma},
\end{equation}
where $E_R$ is the set defined in~\eqref{eq:set_E_R}.
\end{lemma}

\begin{proof}
Fix $R>0$. Split $\{|x|\geq R\}$ into the set $G_R:=\{|x| \geq R\}\setminus E_R$ and the set~$E_R$. On $G_R$ we have $V\ge \kappa|x|^\gamma\ge \kappa R^\gamma$, hence
\[
|T_t(u_k)|^p \le \frac{1}{\kappa R^\gamma}\,V\,|T_t(u_k)|^p
\quad\text{on } G_R.
\]
Integrating and using~\eqref{eq:energy_estimate_Tt(uk)}
gives
\begin{equation} \label{eq:aux1}
\int_{G_R} |T_t(u_k)|^p\,\dx
\le \frac{1}{\kappa R^\gamma}\int_{\R^n} V\,|T_t(u_k)|^p\,\dx
\le \frac{t\|f\|_1}{\kappa R^\gamma}.
\end{equation}
Moreover,
\begin{equation}\label{eq:aux2}
\begin{split}
\int_{|x|>R} \min(|T_t(u_k)|,1)^p\,\dx
 = \ & \int_{E_R} \min(|T_t(u_k)|,1)^p\,\dx\\ & + \int_{G_R} \min(|T_t(u_k)|,1)^p\,\dx \\
 \le \ & |E_R| + \int_{G_R} |T_t(u_k)|^p\,\dx.
\end{split}
\end{equation}
Combining~\eqref{eq:aux1}--\eqref{eq:aux2} yields~\eqref{eq:tail-min-Tt(uk)}.
\end{proof}

\subsection{Stability estimate}

\begin{lemma}\label{lem:stability_approx}
Assume $2 \leq  p < \infty$ and let $\{u_k\}_{k \in \N}$ be the sequence of approximate solutions from Lemma~\ref{lem:approx_uk}. Then, for every $t>0$ and every $k,\ell\in\N$,
\[
\|T_t(u_k-u_\ell)\|_{\X}^p
\le
C_p\,t\,\|f_k-f_\ell\|_{1},
\]
for some constant $C_p>0$ depending only on $p$.
\end{lemma}

\begin{proof}
Fix $t>0$ and $k,\ell\in\N$, and set
\[
w_{k,\ell}:=u_k-u_\ell,
\qquad
\psi:=T_t(w_{k,\ell})\in \X.
\]
Subtracting the weak formulations for $u_k$ and $u_\ell$ and testing with $\psi$,
we obtain
\begin{align*}
\int_{\R^n}
&\Big(|\nabla u_k|^{p-2}\nabla u_k-|\nabla u_\ell|^{p-2}\nabla u_\ell\Big)
\cdot \nabla \psi\,\dx
\\
&\quad
+\int_{\R^n}
V \Big(|u_k|^{p-2}u_k-|u_\ell|^{p-2}u_\ell\Big)\psi\,\dx
=
\int_{\R^n}(f_k-f_\ell)\psi\,\dx .
\end{align*}
Write
\[
A(\xi):=|\xi|^{p-2}\xi,
\qquad
g(s):=|s|^{p-2}s.
\]
Since $p\ge2$, we have
\begin{equation}\label{eq:monotonicity-A}
\big(A(\xi)-A(\eta)\big)\cdot(\xi-\eta)\ge c_p|\xi-\eta|^p
\qquad\text{for all }\xi,\eta\in\R^n,
\end{equation}
with $c_p=2^{2-p}$; see \cite{lindqvist2019stationary}. In particular,
taking $n=1$, we also have
\begin{equation}\label{eq:monotonicity-g}
\big(g(a)-g(b)\big)(a-b)\ge c_p|a-b|^p
\qquad\text{for all }a,b\in\R.
\end{equation}
Using
\[
\nabla\psi=\nabla T_t(w_{k,\ell})
=(\nabla u_k-\nabla u_\ell)\chi_{\{|w_{k,\ell}|<t\}}
\qquad\text{a.e.\ on }\R^n,
\]
we obtain from \eqref{eq:monotonicity-A} that
\[
\int_{\R^n}
\big(A(\nabla u_k)-A(\nabla u_\ell)\big)\cdot \nabla \psi\,\dx
\ge
c_p\int_{\R^n}|\nabla T_t(w_{k,\ell})|^p\,\dx.
\]
For the zero-order term, we claim that
\[
\big(g(a)-g(b)\big)\,T_t(a-b)\ge c_p|T_t(a-b)|^p
\qquad\text{for all }a,b\in\R.
\]
Indeed, if $|a-b|\le t$, then $T_t(a-b)=a-b$ and the claim follows from
\eqref{eq:monotonicity-g}. If $|a-b|>t$, then
\[
T_t(a-b)=t\,\operatorname{sign}(a-b),
\]
and since $g(a)-g(b)$ has the same sign as $a-b$, we have
\[
\big(g(a)-g(b)\big)\,T_t(a-b)=t\,|g(a)-g(b)|.
\]
Moreover,
\[
|g(a)-g(b)|\,|a-b|
=
\big(g(a)-g(b)\big)(a-b)
\ge c_p|a-b|^p,
\]
hence
\[
|g(a)-g(b)|\ge c_p|a-b|^{p-1}\ge c_p t^{p-1},
\]
and therefore
\[
\big(g(a)-g(b)\big)\,T_t(a-b)\ge c_p t^p=c_p|T_t(a-b)|^p.
\]
Applying this pointwise with $a=u_k(x)$ and $b=u_\ell(x)$, we obtain
\[
\int_{\R^n}
V\big(g(u_k)-g(u_\ell)\big)\psi\,\dx
\ge
c_p\int_{\R^n}V\,|T_t(w_{k,\ell})|^p\,\dx.
\]
Finally, since $|\psi|\le t$, we have
\[
\left|\int_{\R^n}(f_k-f_\ell) \, \psi\,\dx\right|
\le
t\,\|f_k-f_\ell\|_1.
\]
Combining the three estimates gives
\[
c_p \int_{\R^n}|\nabla T_t(w_{k,\ell})|^p\,\dx
+c_p \int_{\R^n}V\,|T_t(w_{k,\ell})|^p\,\dx
\le
t\,\|f_k-f_\ell\|_1,
\]
which is exactly the desired estimate with $C_p = c_p^{-1}$.
\end{proof}

\subsection{Localized identity}

\begin{lemma}\label{lem:approx_CTI}
Let $\{u_k\}_{k\in\N}$ be the sequence of approximate solutions from
Lemma~\ref{lem:approx_uk}. Fix $t>0$, let
$\phi\in W^{1,p}(\R^n)\cap L^\infty(\R^n)$ have compact support, and let
$\alpha>t+\|\phi\|_\infty$.
For each $k\in\N$, set
\[
\Phi_k
:=
T_t\big(T_\alpha(u_k)-\phi\big)-T_t\big(T_\alpha(u_k)\big).
\]
Then $\Phi_k\in \X$, $\supp \Phi_k\subseteq \supp\phi$, and
\begin{align}
\int_{\R^n} & |\nabla T_\alpha(u_k)|^{p-2}\nabla T_\alpha(u_k)\cdot \nabla \Phi_k\,\dx
\notag\\
& +\int_{\R^n} V\,|T_\alpha(u_k)|^{p-2}T_\alpha(u_k)\,\Phi_k\,\dx
= \int_{\R^n} f_k\,\Phi_k\,\dx .
\label{eq:approx_compact_identity}
\end{align}
\end{lemma}

\begin{proof}
Since $T_\alpha(u_k)\in \X$, the functions $T_t(T_\alpha(u_k)-\phi)$ and
$T_t(T_\alpha(u_k))$ belong to $W^{1,p}(\R^n)$ by the Sobolev chain rule, and
therefore so does $\Phi_k$. Moreover, $\Phi_k$ is compactly supported because $\phi$ is
compactly supported. Since $V\in L^\infty_{\loc}(\R^n)$ and $|\Phi_k|\le 2t$, it follows
that $\Phi_k\in \X$.

Using $\Phi_k$ as a test function in \eqref{eq:exist_approx_uk}, we obtain
\begin{equation}\label{eq:weak_with_Hk}
\int_{\R^n} |\nabla u_k|^{p-2}\nabla u_k\cdot \nabla \Phi_k\,\dx
+\int_{\R^n} V\,|u_k|^{p-2}u_k\,\Phi_k\,\dx
=
\int_{\R^n} f_k\,\Phi_k\,\dx .
\end{equation}

We now show that the two left-hand side terms may be rewritten with $T_\alpha(u_k)$.

\smallskip
\noindent
\emph{Gradient term.}
By the chain rule,
\[
\nabla \Phi_k
=
T_t'\big(T_\alpha(u_k)-\phi\big)\big(\nabla T_\alpha(u_k)-\nabla\phi\big)
-
T_t'\big(T_\alpha(u_k)\big)\nabla T_\alpha(u_k)
\qquad\text{a.e. in }\R^n.
\]
If $|u_k|>\alpha$, then $T_\alpha(u_k)=\alpha\,\operatorname{sign}(u_k)$ and, by the
assumption on $\alpha$,
\[
|T_\alpha(u_k)-\phi|
\ge \alpha-\|\phi\|_\infty
> t,
\qquad
|T_\alpha(u_k)|=\alpha>t.
\]
Hence both derivatives $T_t'\big(T_\alpha(u_k)-\phi\big)$ and 
$T_t'\big(T_\alpha(u_k)\big)$ vanish a.e.\ on $\{|u_k|>\alpha\}$, and so
\[
\nabla \Phi_k=0
\qquad\text{a.e.\ on }\{|u_k|>\alpha\}.
\]
Since $\nabla T_\alpha(u_k)=\nabla u_k$ a.e.\ on $\{|u_k|<\alpha\}$, we conclude that
\begin{equation}\label{eq:localized_aux1}
|\nabla u_k|^{p-2}\nabla u_k\cdot \nabla \Phi_k
=
|\nabla T_\alpha(u_k)|^{p-2}\nabla T_\alpha(u_k)\cdot \nabla \Phi_k
\qquad\text{a.e.\ in }\R^n.
\end{equation}

\smallskip
\noindent
\emph{Potential term.}
If $|u_k|\le \alpha$, then $T_\alpha(u_k)=u_k$, so
\[
|u_k|^{p-2}u_k\,\Phi_k
=
|T_\alpha(u_k)|^{p-2}T_\alpha(u_k)\,\Phi_k.
\]
On the set $\{|u_k|>\alpha\}$ one has $T_\alpha(u_k)=\alpha\,\operatorname{sign}(u_k)$.
We consider the cases $u_k>\alpha$ and $u_k<-\alpha$ separately. If $u_k>\alpha$, then $T_\alpha(u_k)=\alpha$. Using $|\phi|\le \|\phi\|_\infty$
and the assumption on $\alpha$, we obtain
\begin{equation*}
T_\alpha(u_k)-\phi=\alpha-\phi\ge \alpha-\|\phi\|_\infty>t.
\end{equation*}
Hence
\begin{equation*}
T_t\bigl(T_\alpha(u_k)-\phi\bigr)=t,
\qquad
T_t\bigl(T_\alpha(u_k)\bigr)=T_t(\alpha)=t.
\end{equation*}
Similarly, if $u_k<-\alpha$, then
\begin{equation*}
T_\alpha(u_k)-\phi=-\alpha-\phi\le -\alpha+\|\phi\|_\infty<-t.
\end{equation*}
Therefore
\begin{equation*}
T_t\bigl(T_\alpha(u_k)-\phi\bigr)=-t,
\qquad
T_t\bigl(T_\alpha(u_k)\bigr)=T_t(-\alpha)=-t.
\end{equation*}
In either case,
\begin{equation*}
T_t\bigl(T_\alpha(u_k)-\phi\bigr)
=
t\,\operatorname{sign}(u_k)
=
T_t\bigl(T_\alpha(u_k)\bigr),
\end{equation*}
and so
\[
\Phi_k
=
T_t\bigl(T_\alpha(u_k)-\phi\bigr)-T_t\bigl(T_\alpha(u_k)\bigr)=0.
\]
Consequently,
\begin{equation}\label{eq:localized_aux2}
|u_k|^{p-2}u_k\,\Phi_k
=
|T_\alpha(u_k)|^{p-2}T_\alpha(u_k)\,\Phi_k
\qquad\text{a.e. in }\R^n.
\end{equation}

Substituting~\eqref{eq:localized_aux1} and \eqref{eq:localized_aux2} into \eqref{eq:weak_with_Hk} yields
\eqref{eq:approx_compact_identity} and concludes the proof.
\end{proof}

\section{Compactness and passage to the limit} \label{sec:passage_limit}

In this section, we exploit the compactness theorem proved in
Section~\ref{sec:proof_ARK} in order to pass from the approximate solutions
constructed in Section~\ref{sec:approximate} to a limit object in
$\Lambda^p(\R^n)$. We first show that convergence of the truncations implies
that the approximate sequence is Cauchy in measure. We then prove total
boundedness of the truncated family and perform a diagonal extraction. Finally, we
show that the resulting limit has all truncations in the energy space~$\X$ and
satisfies the corresponding uniform energy bounds.

\subsection{From truncations to convergence in measure}

\begin{lemma} \label{lem:convmeasure_from_trunc}
Let $\{u_k\}_{k \in \N}$ be the sequence of approximate solutions from Lemma~\ref{lem:approx_uk}. If, for each $m \in \N$, the sequence $\{T_m(u_k) \}_{k \in \N}$ converges in measure to some measurable function $v_m$, then the sequence $\{u_k\}_{k \in \N}$ is Cauchy in measure. Additionally, if there exists a measurable function $u$ such that
\begin{equation} \label{eq:aux_appendix}
v_m = T_m(u), \qquad \int_{\R^n}|T_m(u)|^p\,\dx \leq m  \| f\|_1, \qquad \text{for all }m \in \N, 
\end{equation}
 then 
\[u_k \to u \qquad \text{in measure}. \]
\end{lemma}

\begin{proof}
Fix $\varepsilon>0$. Then, for $k,j,m \in \N$,
\[
\{|u_k-u_j|>\varepsilon\}
\subseteq
\{|T_m(u_k)-T_m(u_j)|>\varepsilon\}
\cup
\{|u_k|>m\}
\cup
\{|u_j|>m\}.
\]
Since $\{T_m(u_k)\}_{k \in \N}$ is Cauchy in measure, we have
\[
\big|\{|T_m(u_k)-T_m(u_j)|>\varepsilon\}\big|
\to 0 \qquad \text{as }k,j\to\infty.
\]
On the other hand, by Lemma~\ref{lem:energy_estimate_Tt(uk)} with $t=m$,
\[
\int_{\R^n}|T_m(u_k)|^p\,\dx
\le
\int_{\R^n}V \, |T_m(u_k)|^p\,\dx
\le m \, \|f\|_1 \qquad \text{for all } k \in \N.
\]
Since $|T_m(u_k)|=m$ on $\{|u_k|>m\}$, this gives
\begin{align*}
|\{|u_k|>m\}| &= m^{-p} \int_{|u_k| > m} m^p \, \dx \\
& = m^{-p} \int_{|u_k| > m} |T_m(u_k)|^p \, \dx \\
& \leq m^{1-p}\|f\|_1,
\end{align*}
for all $k \in \N$.

Therefore
\[
\limsup_{k,j\to\infty}
|\{|u_k-u_j|>\varepsilon\}|
\le
2m^{1-p}\|f\|_1.
\]
Letting $m\to\infty$ proves that $\{u_k \}_{k \in \N}$ is Cauchy in measure. \par 
If, in addition, there is some measurable function $u$ such that~\eqref{eq:aux_appendix} holds, then the same argument applied to $\{|u_k-u|>\varepsilon\}$ yields the desired conclusion.
\end{proof}

\subsection{Total boundedness of truncations}

\begin{lemma}\label{lem:ARK_trunc}
Let $\{u_k\}_{k \in \N}$ be the sequence of approximate solutions from Lemma~\ref{lem:approx_uk}. Then, for each $t > 0$ the family $\{T_t(u_k)\}_{k\in\N}$ is totally bounded in $\Lambda^p(\R^n)$.
Consequently, there exist a subsequence $\{u_{k_j^{(t)}}\}_{j\in\N}$ and $v_t \in \Lambda^p(\R^n)$ such that
\begin{equation}\label{eq:ARK_trunc_conv}
  T_t(u_{k_j^{(t)}}) \to v_t \qquad\text{in }\Lambda^p(\R^n) \qquad \text{as }  j \to \infty.
\end{equation}
\end{lemma}

\begin{proof}
Fix $t>0$ and set $\mathcal{F}_t:=\{T_t(u_k) :\, k\in\N\}$.
We verify the hypotheses of Theorem~\ref{thm_ARK} for $\mathcal{F}_t$.

\par
By Lemma~\ref{lem:energy_estimate_Tt(uk)} and $V\ge 1$ a.e.,
\[
  \sup_{k\in\N}\|T_t(u_k)\|_p^p
  \le \sup_{k\in\N}\int_{\R^n}V\,|T_t(u_k)|^p\,\dx
  \le t\,\|f\|_1.
\]
Moreover, Lemma~\ref{lem:energy_estimate_Tt(uk)} also gives
\[
  \sup_{k\in\N}\|\nabla T_t(u_k)\|_{p,\infty} \leq \sup_{k\in\N}\|\nabla T_t(u_k)\|_p <\infty.
\]
Thus Theorem~\ref{thm_ARK}(i) holds with $q=p$.

\par 
Now, let $\varepsilon>0$ and choose $R > 0$ such that
\[|E_R| + \frac{t\|f\|_{1}}{\kappa R^\gamma} < \eps^p.\]
Then Lemma~\ref{lem:tail-measure} implies
\begin{equation}\label{eq:tail_trunc_Lambda}
  \int_{|x|>R} \min\big(|T_t(u_k)|,1\big)^p\,\dx < \varepsilon^p
\end{equation}
for all $k \in \N$. This is Theorem~\ref{thm_ARK}(ii).
\end{proof}

\subsection{Diagonal extraction}
\begin{lemma}\label{lem:diagonal_entropy}
Let $\{u_k\}_{k \in \N}$ be the sequence of approximate solutions of Lemma~\ref{lem:approx_uk}. Then, there exist a subsequence $\{u_{k_j}\}_{j\in\N}$ and a function $u\in \Lambda^p(\R^n)$ such that for every $m\in\N$,
\begin{equation}\label{eq:diag_conv}
  T_m(u_{k_j}) \to T_m(u) \qquad\text{in }\Lambda^p(\R^n)
  \qquad\text{as }j\to\infty.
\end{equation}
\end{lemma}

\begin{proof}

For $m=1$, Lemma~\ref{lem:ARK_trunc} provides a subsequence
$\{u_{k_j^{(1)}}\}_{j\in\N}$ and a function $v_1\in \Lambda^p(\R^n)$ such that
\[
T_1(u_{k_j^{(1)}})\to v_1
\qquad\text{in }\Lambda^p(\R^n).
\]
For $m=2$, apply Lemma~\ref{lem:ARK_trunc} to the subsequence
$\{u_{k_j^{(1)}}\}_{j\in\N}$ to obtain a further subsequence
$\{u_{k_j^{(2)}}\}_{j\in\N}$ and a function $v_2\in \Lambda^p(\R^n)$ such that
\[
T_2(u_{k_j^{(2)}})\to v_2
\qquad\text{in }\Lambda^p(\R^n).
\]
Proceeding inductively, we obtain nested subsequences
$\{u_{k_j^{(m)}}\}_{j\in\N}$ and functions $v_m\in \Lambda^p(\R^n)$ such that
\[
T_m(u_{k_j^{(m)}})\to v_m
\qquad\text{in }\Lambda^p(\R^n)
\]
for every $m\in\N$. 

Let $u_{k_j}:=u_{k_j^{(j)}}$ be the diagonal subsequence. Then, for each fixed $m\in\N$, the tail $\{u_{k_j}\}_{j\ge m}$ is a subsequence of
$\{u_{k_j^{(m)}}\}_{j\in\N}$, hence
\begin{equation}\label{eq:diag_to_vm}
T_m(u_{k_j})\to v_m
\qquad\text{in }\Lambda^p(\R^n)
\qquad\text{as }j\to\infty.
\end{equation}

Since convergence in $\Lambda^p(\R^n)$ implies convergence in measure, using Lemma~\ref{lem:convmeasure_from_trunc} we conclude that sequence $\{u_{k_j}\}_{j\in\N}$ is Cauchy in measure. Then, there exists a further subsequence, still denoted by
$\{u_{k_j}\}_{j\in\N}$, and a real-valued measurable function $u$ on $\R^n$ such that
\[
u_{k_j}(x)\to u(x)
\qquad\text{for a.e.\ }x\in\R^n.
\]
By continuity, for every $m\in\N$,
\[
T_m(u_{k_j}(x))\to T_m(u(x))
\qquad\text{for a.e.\ }x\in\R^n.
\]
Since $T_m(u_{k_j})\to v_m$ in~$\Lambda^p(\R^n)$ and
$T_m(u_{k_j})\to T_m(u)$ pointwise a.e., it follows that
\[
v_m=T_m(u)
\qquad\text{a.e.\ on }\R^n.
\]
Thus, $T_m(u) \in \Lambda^p(\R^n)$ and from \eqref{eq:diag_to_vm} we obtain~\eqref{eq:diag_conv}. Finally, noting that for any $m \in \N$ we have
\[\min(|T_m(u)|,1) = \min(|u|,1), \]
it follows that $u \in \Lambda^p(\R^n)$.
\end{proof}

\subsection{Arbitrary truncation levels}

\begin{lemma}
\label{lem:all-alpha-trunc}
Let $u$ and $\{u_{k_j}\}_{j\in\N}$ be as in Lemma~\ref{lem:diagonal_entropy}.
Then, for every $\alpha>0$,
\begin{equation}\label{eq:all-alpha-Lambda}
T_\alpha(u_{k_j})\to T_\alpha(u)
\qquad\text{in }\Lambda^p(\R^n)
\qquad\text{as }j\to\infty.
\end{equation}
Moreover,
\begin{equation} \label{eq:Talpha(u)inX}
T_\alpha(u)\in \X
\end{equation}
and
\begin{equation}\label{eq:all-alpha-energy}
\|T_\alpha(u) \|_{\X}^p=\int_{\R^n}|\nabla T_\alpha(u)|^p\,\dx
+\int_{\R^n}V \, |T_\alpha(u)|^p\,\dx
\le
\alpha \, \|f\|_{1}.
\end{equation}
\end{lemma}

\begin{proof}
Fix $\alpha>0$ and choose an integer $m>\alpha$. Note that $T_\alpha\circ T_m=T_\alpha$, so
\[
T_\alpha(u_{k_j})=T_\alpha(T_m(u_{k_j})),
\qquad
T_\alpha(u)=T_\alpha(T_m(u)).
\]
Since $T_\alpha:\R\to\R$ is $1$-Lipschitz,
\[
|T_\alpha(a)-T_\alpha(b)|\le |a-b|
\qquad\text{for all } a,b\in\R,
\]
and therefore
\begin{align*}
\min\bigl(|T_\alpha(u_{k_j})-T_\alpha(u)|,1\bigr)& = \min\bigl(|T_\alpha(T_m(u_{k_j}))-T_\alpha(T_m(u))|,1\bigr) \\
& \le
\min\bigl(|T_m(u_{k_j})-T_m(u)|,1\bigr).
\end{align*}
By Lemma~\ref{lem:diagonal_entropy},
\[
T_m(u_{k_j})\to T_m(u)\qquad\text{in }\Lambda^p(\R^n) \qquad \text{as }j \to \infty,
\]
hence
\[
\int_{\R^n}\min\bigl(|T_\alpha(u_{k_j})-T_\alpha(u)|,1\bigr)^p\,\dx\to 0 \qquad \text{as }j \to \infty,
\]
which proves \eqref{eq:all-alpha-Lambda}.

Moreover, since
\[
|T_\alpha(u_{k_j})-T_\alpha(u)|\le 2\alpha,
\]
we have
\[
|T_\alpha(u_{k_j})-T_\alpha(u)|^p
\le
C_\alpha
\min\bigl(|T_\alpha(u_{k_j})-T_\alpha(u)|,1\bigr)^p
\]
with $C_\alpha = \max\{1,(2\alpha)^p\}$. Using \eqref{eq:all-alpha-Lambda}, we get
\begin{equation*} \label{eq:convLp_aux1}
T_\alpha(u_{k_j})\to T_\alpha(u)
\qquad\text{in }L^p(\R^n) \qquad \text{as }j\to \infty.
\end{equation*}

By Lemma~\ref{lem:energy_estimate_Tt(uk)}, the sequence $\{T_\alpha(u_{k_j})\}_{j\in\N}$ is bounded in $\X$. Hence, by Lemma~\ref{lem:weakcompactX}, after passing to a subsequence there exists $w\in \X$ such that
\[
T_\alpha(u_{k_j}) \rightharpoonup w
\qquad\text{weakly in }L^p(\R^n),
\]
and
\[
\|w\|_{\X}\le \liminf_{j\to\infty}\|T_\alpha(u_{k_j})\|_{\X}.
\]
Since $T_\alpha(u_{k_j})\to T_\alpha(u)$ in $L^p(\R^n)$,
it follows that $w=T_\alpha(u)$ a.e.\ in $\R^n$. Thus $T_\alpha(u)\in \X$
and \eqref{eq:all-alpha-energy} follows from the lower semicontinuity estimate and~\eqref{eq:energy_estimate_Tt(uk)}.
\end{proof}

\section{Asymptotic energy solutions} \label{sec:asymptotic_ES}

In this section, we introduce the notion of asymptotic energy solution and prove
Theorem~\ref{thm:AES}. The existence part follows by showing that the limit
function obtained in Section~\ref{sec:passage_limit} satisfies the definition.
We then prove uniqueness by means of the stability estimate from
Section~\ref{sec:approximate}. Finally, we show that, in the duality regime
$f\in L^1(\R^n)\cap L^{p'}(\R^n)$, asymptotic energy solutions are consistent
with the classical weak formulation.

\begin{definition}\label{def:ATS}
Let $f\in L^1(\R^n)$ and let $V:\R^n\to[1,\infty)$ belong to $L^\infty_{\rm loc}(\R^n)$. A function $u\in \Lambda^p(\R^n)$ is called an
\emph{asymptotic energy solution} of~\eqref{eq:p-Schrodinger} provided:
\begin{enumerate}[(i)]
\item For every $\alpha>0$, one has $T_\alpha(u)\in \X$.
\smallskip
\item There exists a sequence $\{f_j\}_{j\in\N}\subseteq L^1(\R^n)\cap L^{p'}(\R^n)$
such that
\[
f_j\to f \qquad\text{in }L^1(\R^n) \qquad\text{as }j\to\infty,
\]
and, if $u_j\in \X$ denotes the unique weak energy solution of~\eqref{eq:p-Schrodinger} with datum~$f_j$, then for every $\alpha>0$,
\[
T_\alpha(u_j-u)\to 0
\qquad\text{in }\X
\qquad\text{as }j\to\infty.
\]
\end{enumerate}
In this case, we say that $\{u_j\}_{j \in \N}$ is an approximate sequence for $u$.
\end{definition}

\subsection{Proof of Theorem~\ref{thm:AES}: Existence}

For each $k \in \N$, let $f_k=T_k(f)\chi_{B_k}$ and let $u_k \in \X$ be the
unique weak energy solution of~\eqref{eq:p-Schrodinger} with right-hand side
$f_k$, according to Lemma~\ref{lem:approx_uk}. By Lemmas~\ref{lem:diagonal_entropy}
and~\ref{lem:all-alpha-trunc}, there exist a subsequence
$\{u_{k_j}\}_{j\in\N}$ and a function $u\in \Lambda^p(\R^n)$ such that
\eqref{eq:all-alpha-Lambda}--\eqref{eq:all-alpha-energy} hold. In particular,
by~\eqref{eq:Talpha(u)inX}, the function $u$ satisfies condition~\textup{(i)} of
Definition~\ref{def:ATS}. Moreover, by Lemma~\ref{lem:convmeasure_from_trunc}, the sequence
$\{u_{k_j}\}_{j\in\N}$ converges to $u$ in measure.

We now prove that, for every $\alpha>0$,
\begin{equation}\label{eq:Talpha(ukj-u)to0}
T_\alpha(u_{k_j}-u)\to 0
\qquad\text{in }\X
\qquad\text{as }j\to\infty.
\end{equation}
Fix $\alpha>0$ and $j\in\N$. By Lemma~\ref{lem:stability_approx},
\begin{equation}\label{eq:aux_existence1}
\|T_\alpha(u_{k_j}-u_{k_\ell})\|_{\X}^p
\le
C_p\,\alpha\,\|f_{k_j}-f_{k_\ell}\|_1
\qquad\text{for all }\ell\in\N.
\end{equation}
Set
\[
w_\ell:=T_\alpha(u_{k_j}-u_{k_\ell}).
\]
From~\eqref{eq:aux_existence1} and the uniform bound $\|f_k \|_1 \leq \|f \|_1$, we see that $\{w_\ell\}_{\ell\in\N}$ is bounded in~$\X$. Then, by Lemma~\ref{lem:weakcompactX}, after passing to a subsequence, still denoted by $\{w_\ell \}_{\ell \in \N}$, there
exists $\widetilde w\in \X$ such that
\[
w_\ell\rightharpoonup \widetilde w
\qquad\text{weakly in }L^p(\R^n),
\]
and
\[
\|\widetilde w\|_{\X}\le \liminf_{\ell\to\infty}\|w_\ell\|_{\X}.
\]

Next, since $T_\alpha:\R\to\R$ is $1$-Lipschitz and $u_{k_\ell}\to u$ in measure,
we claim that
\begin{equation}\label{eq:aux_existence2}
w_\ell\to T_\alpha(u_{k_j}-u)
\qquad\text{in measure}.
\end{equation}
Indeed, for every $\varepsilon>0$,
\begin{align*}
\bigl\{|w_\ell-T_\alpha(u_{k_j}-u)|>\varepsilon\bigr\}
&=
\Bigl\{\bigl|T_\alpha(u_{k_j}-u_{k_\ell})-T_\alpha(u_{k_j}-u)\bigr|>\varepsilon\Bigr\} \\
&\subseteq
\bigl\{|u_{k_\ell}-u|>\varepsilon\bigr\},
\end{align*}
and the right-hand side has measure tending to zero as $\ell\to\infty$.

Now, since $|w_\ell|\le \alpha$ and $|T_\alpha(u_{k_j}-u)|\le \alpha$, we have, for each $\varepsilon>0$ and $R>0$,
\begin{align*}
\int_{B_R}|w_\ell-T_\alpha(u_{k_j}-u)|^p\,\dx
&=
\int_{B_R\cap\{|w_\ell-T_\alpha(u_{k_j}-u)|\le\varepsilon\}}
|w_\ell-T_\alpha(u_{k_j}-u)|^p\,\dx \\
&\quad +
\int_{B_R\cap\{|w_\ell-T_\alpha(u_{k_j}-u)|>\varepsilon\}}
|w_\ell-T_\alpha(u_{k_j}-u)|^p\,\dx \\
&\le
\varepsilon^p |B_R|
+
(2\alpha)^p
\bigl|
B_R\cap\{|w_\ell-T_\alpha(u_{k_j}-u)|>\varepsilon\}
\bigr|.
\end{align*}
By~\eqref{eq:aux_existence2}, the second term tends to zero as $\ell\to\infty$.
Hence
\[
\limsup_{\ell\to\infty}
\int_{B_R}|w_\ell-T_\alpha(u_{k_j}-u)|^p\,\dx
\le \varepsilon^p |B_R|.
\]
Letting $\varepsilon\downarrow0$, we conclude that
\[
w_\ell\to T_\alpha(u_{k_j}-u)
\qquad\text{in }L^p(B_R).
\]

We now identify the weak limit. Since $w_\ell\rightharpoonup \widetilde w$
weakly in $L^p(\R^n)$, it follows by restriction that $w_\ell\rightharpoonup
\widetilde w$ weakly in $L^p(B_R)$. On the other hand, the strong convergence
just proved implies that $w_\ell\rightharpoonup T_\alpha(u_{k_j}-u)$ weakly in
$L^p(B_R)$. By uniqueness of the weak limit,
\[
\widetilde w=T_\alpha(u_{k_j}-u)
\qquad\text{a.e.\ on }B_R.
\]
Since $R>0$ is arbitrary, we conclude that
\[
\widetilde w=T_\alpha(u_{k_j}-u)
\qquad\text{a.e.\ on }\R^n.
\]

It follows that
\[
T_\alpha(u_{k_j}-u)\in \X,
\]
and
\begin{align*}
\|T_\alpha(u_{k_j}-u)\|_{\X}^p
&\le \liminf_{\ell\to\infty}\|w_\ell\|_{\X}^p \\
& \leq \liminf_{\ell\to\infty} C_p \alpha \|f_{k_j} - f_{k_\ell} \|_1.
\end{align*}
Since $f_{k_\ell}\to f$ in $L^1(\R^n)$ and $f_{k_j}$ is fixed, we have
\[
\|f_{k_j}-f_{k_\ell}\|_1 \to \|f_{k_j}-f\|_1
\qquad\text{as }\ell\to\infty.
\]
Combining the previous estimates, we obtain
\[
\|T_\alpha(u_{k_j}-u)\|_{\X}^p
\le
C_p\,\alpha\,\|f_{k_j}-f\|_1 \to 0 \qquad \text{as } j \to \infty,
\]
which proves~\eqref{eq:Talpha(ukj-u)to0}, showing that $u$ is an asymptotic energy solution of~\eqref{eq:p-Schrodinger}.

\qed

\subsection{Proof of Theorem~\ref{thm:AES}: Uniqueness}

Let $u,v\in \Lambda^p(\R^n)$ be two asymptotic energy solutions of
\eqref{eq:p-Schrodinger}. By definition, there exist sequences
\[
g_j,h_j\in L^1(\R^n)\cap L^{p'}(\R^n)
\]
such that
\[
g_j\to f,\qquad h_j\to f
\qquad\text{in }L^1(\R^n),
\]
and if $u_j,v_j\in \X$ denote the corresponding weak energy solutions, then for every
$\alpha>0$,
\begin{equation} \label{eq:aux_uniqueness1}
T_\alpha(u_j-u)\to 0, \qquad T_\alpha(v_j-v)\to 0, 
\qquad\text{in }\X \qquad \text{as } j \to \infty,
\end{equation}

Applying the argument in the proof of Lemma~\ref{lem:stability_approx} to the approximate sequences $\{u_j\}_{j\in \N}$ and $\{v_j\}_{j \in \N}$, we obtain 
\[ \|T_\alpha(u_j-v_j)\|_{\X}^p \le C_p\,\alpha\,\|g_j-h_j\|_1 \qquad\text{for every }j\in\N. \]
 Since $g_j\to f$ and $h_j\to f$ in $L^1(\R^n)$, it follows that \[ T_\alpha(u_j-v_j)\to 0 \qquad\text{in }\X \qquad\text{as }j\to\infty. \] 
 In particular, 
 \begin{equation} \label{eq:aux_uniqueness2}
T_\alpha(u_j-v_j)\to 0 \qquad\text{in measure}.
\end{equation}

On the other hand, \eqref{eq:aux_uniqueness1} implies that, for every $\alpha>0$,
\[
T_\alpha(u_j-u)\to 0,
\qquad
T_\alpha(v_j-v)\to 0
\qquad\text{in measure}.
\]
Now fix $\varepsilon>0$ and choose $\alpha>\varepsilon$.
Then
\[
\{|u_j-u|>\varepsilon\}\subseteq \{|T_\alpha(u_j-u)|>\varepsilon\},
\]
and similarly
\[
\{|v_j-v|>\varepsilon\}\subseteq \{|T_\alpha(v_j-v)|>\varepsilon\}.
\]
It follows that
\[
u_j\to u,
\qquad
v_j\to v
\qquad\text{in measure}.
\]
Hence $u_j-v_j\to u-v$ in measure and, because $T_\alpha$ is $1$-Lipschitz, \[ T_\alpha(u_j-v_j)\to T_\alpha(u-v) \qquad\text{in measure}. \] By uniqueness of the limit, it follows from~\eqref{eq:aux_uniqueness2} that 
\[ T_\alpha(u-v)=0 \qquad\text{a.e.\ on }\R^n. \] 

Since $\alpha>0$ is arbitrary, we conclude that \[ u=v \qquad\text{a.e.\ on }\R^n, \] which finishes the proof of the theorem. \qed

\subsection{Consistency with the weak formulation}

\begin{proposition}\label{prop:Lpprime_upgrade}
Assume $2\le p<\infty$ and let $f\in L^1(\R^n)\cap L^{p'}(\R^n)$.
Then the unique weak energy solution of~\eqref{eq:p-Schrodinger} is the unique
asymptotic energy solution of~\eqref{eq:p-Schrodinger}.
\end{proposition}

\begin{proof}
Let $u\in \X \subseteq \Lambda^p(\R^n)$ be the unique weak energy solution of~\eqref{eq:p-Schrodinger}.
Since $u\in \X$, we have $T_\alpha(u)\in \X$ for every $\alpha>0$. Now take the constant sequences
\[
f_j:=f \qquad\text{and}\qquad u_j:=u
\qquad\text{for all }j\in\N.
\]
Then $\{f_j\}_{j \in \N}$ satisfies the requirements in
Definition~\ref{def:ATS}, and $\{u_j\}_{j \in \N}$ is an approximate sequence for $u$.
Hence, $u$ is an asymptotic energy solution. The uniqueness of asymptotic
energy solutions then yields the conclusion.
\end{proof}

\section{Conditional upgrades of asymptotic energy solutions} \label{sec:cond_upgrades}

In this section, we introduce a stronger localized formulation that is compatible with the approximate problems. We then show that any asymptotic energy solution
can be upgraded to this stronger notion under an additional compactness assumption on the gradients of the truncations. Finally, we prove that, under
suitable local regularity, such a localized asymptotic solution is distributional.

For scalars $\alpha,t>0$ and real-valued measurable functions $v,\phi$ on
$\R^n$, we set
\begin{equation}\label{eq:def_H_alpha_t}
\mathcal H_{\alpha,t}(v,\phi)
\coloneqq
T_t\!\big(T_\alpha(v)-\phi\big)-T_t\!\big(T_\alpha(v)\big).
\end{equation}

\subsection{Localized asymptotic solutions}

\begin{definition}\label{def:LTES_compact}
A function $u\in \Lambda^p(\R^n)$ is called a
\emph{localized asymptotic solution} of~\eqref{eq:p-Schrodinger} provided:
\begin{enumerate}[(i)]
\item For every $\alpha>0$, one has $T_\alpha(u)\in \X$.
\smallskip
\item For every $t>0$, every $\phi\in W^{1,p}(\R^n)\cap L^\infty(\R^n)$
with compact support, and every $\alpha>t + \|\phi\|_{\infty}$, one has
\begin{equation}\label{eq:LTES_compact_identity}
\begin{split}
\int_{\R^n} &
|\nabla T_\alpha(u)|^{p-2}\nabla T_\alpha(u)\cdot
\nabla \mathcal H_{\alpha,t}(u,\phi)\,\dx \\
&+
\int_{\R^n} V\,|T_\alpha(u)|^{p-2}T_\alpha(u)\,
\mathcal H_{\alpha,t}(u,\phi)\,\dx 
= \int_{\R^n} f\,\mathcal H_{\alpha,t}(u,\phi)\,\dx .
\end{split}
\end{equation}
\end{enumerate}
\end{definition}

\begin{remark}\label{rem:support_H_alpha_t}
If $\phi$ has compact support, then $\mathcal H_{\alpha,t}(v,\phi)$ vanishes
outside $\supp \phi$. Indeed, if $x\notin \supp\phi$, then $\phi(x)=0$, and
therefore
\[
\mathcal H_{\alpha,t}(v,\phi)(x)
=
T_t\!\big(T_\alpha(v(x))\big)-T_t\!\big(T_\alpha(v(x))\big)
=0.
\]
Thus, the perturbation in \eqref{eq:LTES_compact_identity} is compactly
supported.
\end{remark}

\begin{theorem}\label{thm:conditional_upgrade}
Assume $p\ge2$. Let $u\in \Lambda^p(\R^n)$ be the asymptotic energy solution
of~\eqref{eq:p-Schrodinger}, and let $\{u_j\}_{j\in\N}\subseteq \X$ be an
approximate sequence for $u$, with corresponding data
$\{f_j\}_{j\in\N}\subseteq L^1(\R^n)\cap L^{p'}(\R^n)$. Assume, in addition,
that for every $\alpha>0$,
\begin{equation}\label{eq:strong_local_grad_assumption}
\nabla T_\alpha(u_j)\to \nabla T_\alpha(u)
\qquad\text{in }L^p_{\loc}(\R^n;\R^n)
\qquad\text{as }j\to\infty.
\end{equation}
Then $u$ is a localized asymptotic solution of~\eqref{eq:p-Schrodinger}.
\end{theorem}

\begin{proof}
We show that $u$ satisfies Definition~\ref{def:LTES_compact}. 
By Definition~\ref{def:ATS}, one has $T_\alpha(u)\in \X$ for all $\alpha > 0$.

Fix $t>0$, let $\phi\in W^{1,p}(\R^n)\cap L^\infty(\R^n)$ be compactly supported,
and assume $\alpha>t + \|\phi\|_\infty$. Choose $R>0$ such that $\supp\phi\subseteq B_R$ and set
\[
\Phi_j:=\mathcal H_{\alpha,t}(u_j,\phi),
\qquad
\Phi:=\mathcal H_{\alpha,t}(u,\phi).
\]
Both $\Phi_j$ and $\Phi$ vanish outside
$B_R$, and
\[
|\Phi_j|\le 2t,
\qquad
|\Phi|\le 2t
\qquad\text{on }\R^n.
\]

Since $\{u_j\}_{j\in\N}$ is an approximate sequence for $u$, we have
\[
T_{2\alpha}(u_j-u)\to0
\qquad\text{in }\X
\qquad\text{as }j\to\infty,
\]
hence also in $L^p(\R^n)$. Using the inequality
\[
|T_\alpha(a)-T_\alpha(b)| \leq  \min(|a-b|,2\alpha) = |T_{2\alpha}(a-b)|
\qquad\text{for all }a,b\in\R,
\]
we infer that
\[
T_\alpha(u_j)\to T_\alpha(u)
\qquad\text{in }L^p(\R^n).
\]
Together with assumption~\eqref{eq:strong_local_grad_assumption}, this gives
\begin{equation} \label{eq:LAS_aux1}
T_\alpha(u_j)\to T_\alpha(u)
\qquad\text{in }W^{1,p}(B_R).
\end{equation}

Next, we claim that the map
\[
v\mapsto T_t(v-\phi)-T_t(v)
\]
is continuous from $W^{1,p}(B_R)$ to itself. It is enough to show that
$w\mapsto T_t(w)$ is continuous on $W^{1,p}(B_R)$. If
\[
w_j\to w \qquad\text{in }W^{1,p}(B_R),
\]
then, by the Lipschitz continuity of $T_t$,
\[
T_t(w_j)\to T_t(w)
\qquad\text{in }L^p(B_R).
\]
Moreover, by the chain rule,
\[
\nabla T_t(w_j)=\chi_{\{|w_j|<t\}}\nabla w_j,
\qquad
\nabla T_t(w)=\chi_{\{|w|<t\}}\nabla w
\qquad\text{a.e.\ in }B_R,
\]
and on the level set $\{|w|=t\}$ one has $\nabla w=0$ a.e.\ by \cite[Theorem~6.19]{lieb2001analysis}. Hence
\begin{align*}
\nabla T_t(w_j)-\nabla T_t(w)
&=
\chi_{\{|w_j|<t\}}(\nabla w_j-\nabla w)
+
\bigl(\chi_{\{|w_j|<t\}}-\chi_{\{|w|<t\}}\bigr)\nabla w.
\end{align*}
The first term converges to $0$ in $L^p(B_R)$. For the second one, after
passing to a subsequence we may assume that $w_j\to w$ a.e.\ in $B_R$, so
\[
\chi_{\{|w_j|<t\}}\to \chi_{\{|w|<t\}}
\qquad\text{a.e. on }\{|w|\neq t\},
\]
while $\nabla w=0$ a.e.\ on $\{|w|=t\}$. Therefore
\[
\bigl(\chi_{\{|w_j|<t\}}-\chi_{\{|w|<t\}}\bigr)\nabla w \to 0
\qquad\text{a.e. in }B_R,
\]
and dominated convergence yields convergence to $0$ in $L^p(B_R)$. Thus
\[
T_t(w_j)\to T_t(w)
\qquad\text{in }W^{1,p}(B_R),
\]
which proves the claim.

It follows from~\eqref{eq:LAS_aux1} that
\begin{equation}\label{eq:Phi_j_to_Phi}
\Phi_j\to \Phi
\qquad\text{in }W^{1,p}(B_R)
\qquad\text{as }j\to\infty.
\end{equation}

Now, set
\[
A_j:=|\nabla T_\alpha(u_j)|^{p-2}\nabla T_\alpha(u_j),
\qquad
A:=|\nabla T_\alpha(u)|^{p-2}\nabla T_\alpha(u).
\]
Using the inequality
\[\big||\xi|^{p-2}\xi - |\eta|^{p-2}\eta\big| \leq (p-1) (|\xi| + |\eta|)^{p-2}|\xi - \eta| \qquad \text{for all }\xi,\eta \in \R^n, \]
which is valid for $p \geq 2$; cf.~\cite{lindqvist2019stationary}, we see that the map
\[
\xi\mapsto |\xi|^{p-2}\xi
\]
is continuous from $L^p(B_R;\R^n)$ to $L^{p'}(B_R;\R^n)$. This, together with assumption~\eqref{eq:strong_local_grad_assumption}, yields
\begin{equation}\label{eq:flux_conv}
A_j\to A
\qquad\text{in }L^{p'}(B_R;\R^n)
\qquad\text{as }j\to\infty.
\end{equation}

For each $j\in\N$, Lemma~\ref{lem:approx_CTI} gives
\begin{equation}\label{eq:localized_identity_j}
\int_{\R^n} A_j\cdot \nabla\Phi_j\,\dx
+
\int_{\R^n} V\,g_j\,\Phi_j\,\dx
=
\int_{\R^n} f_j\,\Phi_j\,\dx,
\end{equation}
where
\[
g_j:=|T_\alpha(u_j)|^{p-2}T_\alpha(u_j).
\]
Set also
\[
g:=|T_\alpha(u)|^{p-2}T_\alpha(u).
\]

For the gradient term, we write
\begin{align*}
\left|
\int_{\R^n} A_j\cdot \nabla\Phi_j\,\dx
-
\int_{\R^n} A\cdot \nabla\Phi\,\dx
\right|
&\le
\|A_j-A\|_{L^{p'}(B_R)}\,\|\nabla\Phi_j\|_{L^p(B_R)} \\
&\quad+
\|A\|_{L^{p'}(B_R)}\,\|\nabla\Phi_j-\nabla\Phi\|_{L^p(B_R)}.
\end{align*}
By \eqref{eq:Phi_j_to_Phi} and \eqref{eq:flux_conv}, both terms on the right-hand side tend to $0$.

For the zero-order term, since $p\ge2$, the scalar map
\[
s\mapsto |s|^{p-2}s
\]
is Lipschitz on $[-\alpha,\alpha]$. Since
\[
T_\alpha(u_j)\to T_\alpha(u)
\qquad\text{in }L^p(B_R),
\]
it follows that
\[
g_j\to g
\qquad\text{in }L^p(B_R),
\]
and therefore also in $L^1(B_R)$. Since $V\in L^\infty(B_R)$,
$|g|\le \alpha^{p-1}$, and $|\Phi_j|\le 2t$, we obtain
\begin{align*}
\Big|
\int_{\R^n} V\,g_j\,\Phi_j\,\dx
& -
\int_{\R^n} V\,g\,\Phi\,\dx
\Big| \\
&\le
\|V\|_{L^\infty(B_R)}
\left(
2t\,\|g_j-g\|_{L^1(B_R)}
+
\alpha^{p-1}\|\Phi_j-\Phi\|_{L^1(B_R)} \right),
\end{align*}
which tends to $0$ by the convergence of $\{g_j \}_{j\in\N}$ to $g$ in $L^1(B_R)$ and by~\eqref{eq:Phi_j_to_Phi}.

For the right-hand side of~\eqref{eq:localized_identity_j}, we write
\[
\int_{\R^n} f_j\Phi_j\,\dx-\int_{\R^n}f\Phi\,\dx
=
\int_{\R^n}(f_j-f)\Phi_j\,\dx
+
\int_{\R^n}f(\Phi_j-\Phi)\,\dx .
\]
Since $|\Phi_j|\le 2t$, the first term satisfies
\[
\left|\int_{\R^n}(f_j-f)\Phi_j\,\dx\right|
\le 2t\,\|f_j-f\|_1
\to 0.
\]
For the second term, passing to a subsequence if necessary, we may assume that
\[
\Phi_j\to \Phi
\qquad\text{a.e.\ on }B_R.
\]
Since both $\Phi_j$ and $\Phi$ vanish outside $B_R$, it follows that
\[
\Phi_j\to \Phi
\qquad\text{a.e.\ on }\R^n.
\]
Moreover, using $|\Phi_j|\le 2t$ and $|\Phi|\le 2t$, we have
\[
|f(\Phi_j-\Phi)|\le 4t\,|f|
\qquad\text{a.e. on }\R^n.
\]
Because $f\in L^1(\R^n)$, dominated convergence yields
\[
\int_{\R^n}f(\Phi_j-\Phi)\,\dx\to 0.
\]
Therefore
\[
\int_{\R^n} f_j\Phi_j\,\dx
\to
\int_{\R^n} f\Phi\,\dx .
\]
Passing to the limit in \eqref{eq:localized_identity_j}, we obtain
\[
\int_{\R^n}
|\nabla T_\alpha(u)|^{p-2}\nabla T_\alpha(u)\cdot\nabla \Phi\,\dx
+
\int_{\R^n}
V\,|T_\alpha(u)|^{p-2}T_\alpha(u)\,\Phi\,\dx
=
\int_{\R^n} f\,\Phi\,\dx .
\]
This completes the proof.
\end{proof}

\begin{remark}
Theorem~\ref{thm:conditional_upgrade} shows that the only missing ingredient for
upgrading an asymptotic energy solution to a localized asymptotic solution is the
identification of the nonlinear gradient term. A sufficient condition is the strong local convergence of the gradients of the truncations along one approximate sequence for the solution. This is analogous to the role played by gradient convergence in classical approximation schemes for nonlinear elliptic equations; see, for instance,~\cite{boccardo1992gradients}. More broadly, the
need for such additional gradient control is related to the modern theory of
gradient estimates for quasilinear elliptic equations with rough data; see, for instance,~\cite{byun2022gradient, dong2024gradient}.
\end{remark}

\subsection{Consistency with the distributional formulation}

\begin{proposition}\label{prop:LTES_implies_distributional}
Let $u$ be a localized asymptotic solution of~\eqref{eq:p-Schrodinger} such that
\[
u\in W^{1,p}_{\loc}(\R^n)\cap L^\infty_{\loc}(\R^n).\]
Then $u$ is a distributional solution of~\eqref{eq:p-Schrodinger}, i.e.
\[
\int_{\R^n}
|\nabla u|^{p-2}\nabla u\cdot \nabla\psi\,\dx
+
\int_{\R^n}
V\,|u|^{p-2}u\,\psi\,\dx
=
\int_{\R^n} f\,\psi\,\dx,
\]
for all $\psi\in C_c^\infty(\R^n)$.
\end{proposition}

\begin{proof}
Fix $\psi\in C_c^\infty(\R^n)$ and choose $R>0$ such that
\[
\supp\psi\subseteq B_R.
\]
Let $\lambda\in\R\setminus\{0\}$. Since $u\in L^\infty(B_R)$, we may choose $t>0$
such that
\[
t>\|u\|_{L^\infty(B_R)}+|\lambda|\,\|\psi\|_\infty.
\]
Then choose $\alpha>0$ such that
\[
\alpha>t+|\lambda|\,\|\psi\|_\infty.
\]
Set
\[
\phi:=\lambda\psi.
\]
Then $\phi\in W^{1,p}(\R^n)\cap L^\infty(\R^n)$ has compact support and
$\alpha>t + \|\phi\|_\infty$, so \eqref{eq:LTES_compact_identity} applies.

On $B_R$ we have
\[
|u|<t<\alpha
\qquad\text{and}\qquad
|u-\lambda\psi|
\le
\|u\|_{L^\infty(B_R)}+|\lambda|\,\|\psi\|_\infty
<t.
\]
Hence, on $B_R$,
\[
T_\alpha(u)=u,
\qquad
\nabla T_\alpha(u)=\nabla u
\qquad\text{a.e.},
\]
and
\[
T_t\!\big(T_\alpha(u)-\phi\big)=T_t(u-\lambda\psi)=u-\lambda\psi,
\qquad
T_t\!\big(T_\alpha(u)\big)=T_t(u)=u.
\]
Thus
\[
\mathcal H_{\alpha,t}(u,\phi)=-\lambda\psi
\qquad\text{on }B_R.
\]
Outside $\supp\psi$ one has $\phi=0$, hence
\[
\mathcal H_{\alpha,t}(u,\phi)=0.
\]
Therefore
\[
\mathcal H_{\alpha,t}(u,\phi)=-\lambda\psi
\qquad\text{a.e. on }\R^n.
\]

Substituting this into \eqref{eq:LTES_compact_identity} and using
$T_\alpha(u)=u$ and $\nabla T_\alpha(u)=\nabla u$ on $\supp\psi$, we obtain
\[
-\lambda\int_{\R^n}
|\nabla u|^{p-2}\nabla u\cdot \nabla\psi\,\dx
-\lambda\int_{\R^n}
V\,|u|^{p-2}u\,\psi\,\dx
=
-\lambda\int_{\R^n}
f\,\psi\,\dx .
\]
Dividing by $-\lambda$ yields the desired distributional identity.
\end{proof}

\appendix
\section{Some facts on asymptotic $L^p$ spaces} \label{appendix}

Here we briefly recall the definition of asymptotic $L^p$ spaces on a general measure space $(X,\Sigma,\mu)$, denoted by $\Lambda^p(X)$ for $1\le p<\infty$. We record a few basic facts, including simple inclusion results and their relation with weak Lebesgue spaces.

The asymptotic $L^p$ spaces considered here originated in~\cite{alves2025F} by endowing the class of real-valued measurable functions that are \emph{almost in $L^p$} with the $\F$-norm $\|\!\min(|\idot|,1)\|_p$. More precisely, a function $f$ is almost in $L^p$ if for every $\delta>0$ there exists $E_\delta\in\Sigma$ with $\mu(E_\delta)<\delta$ such that $f\,\chi_{E_\delta^c}\in L^p(X)$. Equipped with this $\F$-norm, the space $\Lambda^p(X)$ becomes a completely metrizable topological vector space containing the classical space $L^p(X)$ as a dense subspace. This $\F$-norm generates the topology of asymptotic $L^p$-convergence and, on finite measure spaces, is equivalent to convergence in measure; see~\cite{alves2024mode, alves2024relation, alves2025F}. Moreover, several classical convergence results, such as dominated and Vitali convergence theorems, admit analogs in this framework~\cite{alves2025F}.

 The first result of this appendix characterizes the space
$\Lambda^p(X)$ in two different ways and, in particular, justifies the
definition~\eqref{eq:spaceLambda}.

\begin{proposition}
Fix $1 \leq p < \infty$, and let $f$ be a real-valued measurable function on $X$. The following statements are equivalent:
\begin{enumerate}[(i)]
\item $f$ is almost in $L^p$.
\smallskip
\item There exists a sequence $\{f_k \}_{k\in\N}$ in $L^p(X)$ such that 
\[\|\!\min(|f_k-f|,1)\|_p \to 0 \qquad \text{as} \ k \to \infty. \]
\item $\min(|f|,1) \in L^p(X)$.
\end{enumerate}
In particular, if $\mu(X) < \infty$ then $\Lambda^p(X)$ consists of all real-valued measurable functions on~$X$ with the topology of convergence in measure.
\end{proposition}
\begin{proof}
~\\
\emph{$(i) \Rightarrow (ii)$}. For each $k \in \N$, let $E_k \in \Sigma$ be such that $\mu(E_k) < 1/k$ and $f_k \coloneqq f \chi_{E_k^c} \in L^p(X)$. Then 
\begin{align*}
\|\! \min(|f_k-f|,1) \|_p^p & = \int_{E_k} \min(|f \chi_{E_k^c}-f|,1)^p \, \dm \\
& \leq \mu(E_k) < \frac{1}{k} \to 0 \qquad \text{as } k \to \infty.
\end{align*}
\emph{$(ii) \Rightarrow (iii)$}. Let $N \in \N$ be such that $\|\! \min(|f_N-f|,1) \|_p < 1$. Then, by the triangle inequality,
\[\|\! \min(|f|,1) \|_p \leq \|\! \min(|f-f_N|,1) \|_p + \|\! \min(|f_N|,1) \|_p < 1 + \|f_N \|_p < \infty. \]
\emph{$(iii) \Rightarrow (i)$}. For each $k \in \N$, let $E_k = \{|f| > k\}$. Then $E_{k+1} \subseteq E_k$ for all $k \in \N$, and by hypothesis we also have $\mu(E_1) < \infty$. Hence
\[\lim_{k \to \infty} \mu(E_k) = \mu \left( \bigcap_{k \in \N} E_k\right) = 0, \]
otherwise $f$ would be infinite on a set of positive measure, which is impossible as $f$ is (a.e.) real-valued. 

Thus, given $\delta > 0$ we may choose $K_{\delta} \in \N$ so that $\mu(E_{K_\delta})<\delta$. On $E_{K_\delta}^c$ we have 
\[|f|^p \leq K_\delta^p \min(|f|,1)^p, \]
which gives $f \chi_{E_{K_\delta}^c} \in L^p(X)$ and finishes the proof.
\end{proof}

The next result provides the natural nesting property of the asymptotic $L^p$ spaces, independently of whether $X$ has finite or infinite measure.

\begin{proposition}
For all $1 \leq p \leq q < \infty$, the following inclusion holds 
\[\Lambda^p(X) \subseteq \Lambda^q(X). \]
\end{proposition}
\begin{proof}
Simply note that for $1 \leq p \leq q < \infty$ one has $\min(|\idot|,1)^q \leq \min(|\idot|,1)^p$.
\end{proof}

Our last result compares weak Lebesgue spaces with asymptotic $L^p$ spaces for
different exponents. When $p=q$ and $\mu(X)=\infty$, there are examples showing
that $L^{p,\infty}(X)$ and $\Lambda^p(X)$ are not comparable, in the sense that
neither is contained in the other~\cite{alves2024relation}.

\begin{proposition}
Assume that $(X,\Sigma,\mu)$ is $\sigma$-finite. Then, for $1 \leq p < q < \infty$,
\[L^{p,\infty}(X) \subseteq \Lambda^q(X). \]
\end{proposition}

\begin{proof}
Let $f \in L^{p,\infty}(X)$. Since $0\le \min(|f|,1)\le 1$, the layer-cake representation yields
\begin{align*}
\int_X \min(|f(x)|,1)^q \, \dm
&= q \int_0^1 t^{q-1}\mu\bigl(\{\min(|f|,1)>t\}\bigr)\,\dt \\
&\le q \int_0^1 t^{q-1}\mu\bigl(\{|f|>t\}\bigr)\,\dt \\
&\le q \int_0^1 t^{q-1-p}\|f\|_{p,\infty}^p\,\dt \\
&= \frac{q}{q-p}\|f\|_{p,\infty}^p.
\end{align*}
Hence $\min(|f|,1)\in L^q(X)$, and therefore $f\in\Lambda^q(X)$.
\end{proof}

\medskip

{\small
\section*{Acknowledgments}
This publication is based upon work supported by King Abdullah University of Science and Technology (KAUST) under Award No. ORFS-CRG12-2024-6430. This work was initiated while N.~J.~Alves was a postdoctoral researcher at the University of Vienna, supported by the Austrian Science Fund (FWF), project number 10.55776/F65.}

\end{document}